\documentclass[12pt]{amsart}
\usepackage{a4wide}
\usepackage{amssymb}
\usepackage{amsmath}
\usepackage{amsfonts}
\usepackage{amsthm}
\usepackage{graphicx}
\usepackage{epsfig}
\usepackage{longtable}
\usepackage{paralist}
\usepackage[usenames,dvipsnames]{color}
\usepackage{hyperref}
\usepackage{caption}
\usepackage[labelformat=simple,labelfont={}]{subfig}

\newtheorem{thm}{Theorem}
\newtheorem{lem}[thm]{Lemma}
\newtheorem{cor}[thm]{Corollary}

\theoremstyle{definition}

\theoremstyle{remark}

\newtheorem*{rmks}{Remarks}

\DeclareMathOperator{\betaIregularized}{I}
\DeclareMathOperator{\betaB}{B}

\DeclareMathOperator{\DEF}{:=}

\newcommand{\PT}[1]{\mathbf{#1}}

\DeclareMathOperator{\dd}{\mathrm{d}}
\DeclareMathOperator{\xctint}{I}
\DeclareMathOperator{\numint}{Q}

\DeclareMathOperator{\bernoulliB}{B}

\DeclareMathOperator{\betafcn}{B}
\DeclareMathOperator{\incompletebetafcnregularized}{I}

\DeclareMathOperator{\diam}{diam}

\DeclareMathOperator{\gammafcn}{\Gamma}
\DeclareMathOperator{\gammafcnregularizedP}{P}
\DeclareMathOperator{\gammafcnregularizedQ}{Q}

\DeclareMathOperator{\indicatorfunction}{\mathbf{1}}

\DeclareMathOperator{\EulerE}{\mathrm{e}}

\DeclareMathOperator{\inversegammafcnregularizedQ}{Q^{-1}}

\DeclareMathOperator{\minwce}{minwce}

\DeclareMathOperator{\psidensity}{f_\psi}

\DeclareMathOperator{\RK}{K}
\DeclareMathOperator{\RKmod}{\mathcal{K}}

\DeclareMathOperator{\wce}{wce}

\DeclareMathOperator{\HyperF}{F}

\newcommand{\Hypergeom}[5]{{\sideset{_#1}{_#2}\HyperF\!\left(\substack{\displaystyle#3\\\displaystyle#4};#5\right)}}

\allowdisplaybreaks[1]

\title[Spatial low-discrepancy sequences]{Spatial low-discrepancy sequences, spherical cone discrepancy, and applications in financial modeling}
\author[J. S. Brauchart, J. Dick, and L. Fang]{Johann S. Brauchart, Josef Dick and Lou Fang} 
\thanks{\noindent The research of the first two authors was supported under Australian Research Council's Discovery Projects funding scheme (project number DP120101816).}

\date{\today}

\DeclareCaptionLabelFormat{andtable}{#1˜\ #2 and \tablename˜\ \thetable}

\begin{document}

\address{ \tiny
J. S. Brauchart and J. Dick:
School of Mathematics and Statistics,
University of New South Wales,
Sydney, NSW, 2052,
Australia }
\email{j.brauchart@unsw.edu.au, josef.dick@unsw.edu.au}

\address{ \tiny
Lou Fang: Department of Mathematical Sciences, Tsinghua University,
Beijing, 100084, China
} \email{louf@163.com}

\begin{abstract}
In this paper we introduce a reproducing kernel Hilbert space defined on~$\mathbb{R}^{d+1}$  as the tensor product of a reproducing kernel defined on the unit sphere $\mathbb{S}^{d}$ in $\mathbb{R}^{d+1}$ and a reproducing kernel defined on $[0,\infty)$. We extend Stolarsky's invariance principle to this case and prove upper and lower bounds for numerical integration in the corresponding reproducing kernel Hilbert space.

The idea of separating the direction from the distance from the origin can also be applied to the construction of quadrature methods. An extension of the area-preserving Lambert transform is used to generate points on $\mathbb{S}^{d-1}$ via lifting Sobol' points in $[0,1)^{d}$ to the sphere. The $d$-th component of each Sobol' point, suitably transformed, provides the distance information so that the resulting point set is normally distributed in $\mathbb{R}^{d}$.

Numerical tests provide evidence of the usefulness of constructing Quasi-Monte Carlo type methods for integration in such spaces. We also test this method on examples from financial applications (option pricing problems) and compare the results with traditional methods for numerical integration in $\mathbb{R}^{d}$.
\end{abstract}


\maketitle


\section{Introduction}

We study numerical integration of functions defined in $\mathbb{R}^{d+1}$ for $d \geq 0$,
\begin{equation}\label{eq_int}
\xctint[\psi](f) \DEF \int_{\mathbb{R}^{d+1}} f(\PT{x}) \psi(\PT{x}) \dd \lambda_{d+1}(\PT{x}),
\end{equation}
where $\psi$ is a probability density function (pdf) (typically a normal or related distribution) and $\lambda_{d+1}$ is the Lebesgue measure on $\mathbb{R}^{d+1}$, by means of Quasi-Monte Carlo (QMC) methods
\begin{equation} \label{eq:num.int}
\numint[X_N](f) \DEF \frac{1}{N} \sum_{j=1}^N f(\PT{x}_j).
\end{equation}
These methods are exact for constant functions. The requirement that $\numint[X_N](f) \to \xctint[\psi](f)$ as $N \to \infty$ for every continuous function defined in $\mathbb{R}^{d+1}$ imposes the condition that the quadrature nodes $\PT{x}_1,\ldots, \PT{x}_N$ have limit distribution given by the pdf $\psi$. A standard method for generating low-discrepancy sequences of quadrature points with the required distribution is by using low-discrepancy points in $[0,1)^{d+1}$ and mapping them to $\mathbb{R}^{d+1}$ via the inverse cumulative distribution function (cdf) of $\psi$, provided the inverse cdf is known. Here we use the following approach: starting with a sequence of well-distributed point sets on the unit sphere $\mathbb{S}^d \DEF \{\PT{x} \in \mathbb{R}^{d+1}: \|\PT{x}\| = 1\}$, we then change the radii of the points such that the resulting configurations in $\mathbb{R}^{d+1}$ follow the required distribution.

%
The analysis of the performance of our integration strategy makes use of the reproducing kernel Hilbert space framework by assuming that the functions to be integrated are from a certain reproducing kernel Hilbert space $\mathcal{H}(\RK)$ defined over $\mathbb{R}^{d+1}$. An essential tool will be an explicit expression for the worst-case integration error of our QMC methods in terms of the kernel $\RK$. Our kernel construction leads to a geometrical interpretation of the worst-case error as an $\mathbb{L}_2$-discrepancy of the integration nodes with respect to test sets that are truncated infinite (anchored at infinity) spherical cones. The underlying relation gives rise to an invariance principle (cf. Theorem~\ref{thm:wce.form.A}) similar to Stolarsky's invariance principle for the sphere; cf. \cite{BrDi2013b,BrDi2013} and \cite{St1973}.
%
%
%
We define the reproducing kernel $\RK$ geometrically as follows. A \emph{spherical cap} with center $\PT{z}^* \in \mathbb{S}^d$ and height $t \in [-1,1]$ is the set
\begin{equation*}
\mathcal{C}( \PT{z}^*, t ) \DEF \left\{ \PT{y}^* \in \mathbb{S}^d : \PT{y}^* \cdot \PT{z}^* \geq t \right\}.
\end{equation*}
Based on a spherical cap $\mathcal{C}( \PT{z}^*, t )$ we define the \emph{truncated infinite spherical cone} as the set
\begin{equation*}
\mathcal{C}( \PT{z}^*, t; R ) \DEF \left\{ \rho \PT{y}^* \in \mathbb{R}^{d+1} : \PT{y}^* \in \mathcal{C}( \PT{z}^*, t ), \rho \geq R  \right\}.
\end{equation*}
For $R=0$ the set $\mathcal{C}(\PT{z}^*, t; R)$ is an infinite spherical cone and for $R > 0$ it is the intersection of the infinite spherical cone with the complement of an open ball of radius $R$ centered at the origin. 
Let $\phi:[0,\infty) \to [0,\infty)$ be a probability density function. Then we set 
\begin{equation*}
\RK(\PT{x}, \PT{y}) \DEF \int_0^\infty \int_{-1}^1 \int_{\mathbb{S}^d} \indicatorfunction_{\mathcal{C}(\PT{z}^*, t; R)}(\PT{x}) \indicatorfunction_{\mathcal{C}(\PT{z}^*, t; R)}(\PT{y}) \dd \sigma_d(\PT{z}^*) \dd t \, \phi(R) \dd R, \qquad \PT{x}, \PT{y} \in \mathbb{R}^{d+1},
\end{equation*}
where $\indicatorfunction_{A}$ is the indicator function of the set $A$.
In Section~\ref{sec:rkhs} we derive this kernel as a product of a kernel defined on $[0,\infty) \times [0,\infty)$ denoted by $\RK_{\mathcal{R}}$ and a kernel $\RK_{\mathcal{S}}$ defined on $\mathbb{S}^d \times \mathbb{S}^d$. Thus the corresponding reproducing kernel Hilbert space $\mathcal{H}(\RK)$ is a tensor product of a reproducing kernel Hilbert space $\mathcal{H}(\RK_{\mathcal{S}})$ defined on $\mathbb{S}^d$ and a reproducing kernel Hilbert space $\mathcal{H}(\RK_{\mathcal{R}})$ defined on $[0,\infty)$.

Note that $\RK(\PT{x}, \PT{y})$ assumes the value $0$ whenever at least one of the arguments $\PT{x}, \PT{y}$ is $\PT{0}$. Thus the kernel is anchored at $\PT{0}$. This implies that all functions in $\mathcal{H}(\RK)$ vanish at $\PT{0}$. 
Assume now that we are given a function $f$ with $f(\PT{0}) = C$ for some constant $C \in \mathbb{R}$ and that $f-C \in \mathcal{H}(\RK)$. Then
\begin{equation*}
\numint[X_N](f) - \xctint[\psi](f) = \numint[X_N](f-C) - \xctint[\psi](f-C),
\end{equation*}
since constant functions are integrated exactly by $\numint[X_N]$. Thus results for the worst-case error $\wce( \numint[X_N]; \mathcal{H}(\RK))$ apply also to functions $f$ such that $f-f(\PT{0}) \in \mathcal{H}(\RK)$. In other words, the restriction that $f(\PT{0}) = 0$ for all $f \in \mathcal{H}(\RK)$ can be removed when discussing numerical integration using QMC methods.

%
%
Also note that, in general, $\phi$ and $\psi$ are not related and can be chosen independently, provided that \eqref{eq_int} is well defined.


To prove upper bounds on the integration error, we study the QMC mean, that is, the average over all possible choices of quadrature points (which have the correct distribution). This shows that a typical QMC method with nodes that are selected at random independently and identically $\psi \lambda_{d+1}$-distributed in $\mathbb{R}^{d+1}$ achieve an upper bound of the order $N^{-1/2}$ (Theorem~\ref{thm:rms.wce.form.A}).

To prove a lower bound for the worst-case error, we show that $\mathcal{H}(\RK_{\mathcal{S}})$ and $\mathcal{H}(\RK_{\mathcal{R}})$ are isomorphically embedded in $\mathcal{H}(\RK)$. Thus, known lower bounds for numerical integration in $\mathcal{H}(\RK_{\mathcal{S}})$ provide lower bounds for numerical integration in $\mathcal{H}(\RK)$. The lower bound is of order $N^{-1/2 - 1/(2d)}$ (Theorem~\ref{thm:wce.lower.bound}).


We also present numerical results for a trial function on the sphere and three problems from option pricing. We compare our method with standard Monte Carlo and Quasi-Monte Carlo approaches. We observe that our method performs better than the Monte Carlo simulation and on average marginally better than the QMC approach.


\section{Spherical cone discrepancy }\label{sec_cone_dis}

In the following let $\sigma_d$ be the normalized surface area measure on the unit sphere $\mathbb{S}^d$ in~$\mathbb{R}^{d+1}$. (The non-normalized surface of the sphere is then denoted by $\omega_d$.)

\subsection{A reproducing kernel Hilbert space on $\mathbb{R}^{d+1}$}
\label{sec:rkhs}

In what follows we introduce a reproducing kernel Hilbert space on $\mathbb{R}^{d+1}$ as a tensor product space of two reproducing kernel Hilbert spaces. The motivation comes from the fact that every point $\PT{x} \in \mathbb{R}^{d+1} \setminus \{\PT{0}\}$ can be decomposed into a direction (represented by a point $\PT{x}^* \in \mathbb{S}^d$) and the distance to the origin ($r > 0$); that is, $\PT{x} = r \, \PT{x}^*$.
%

Let ${\phi: [0, \infty) \to \mathbb{R}}$ be a probability density function. We define the following  kernel $\RK_{\mathcal{R}}: [0,\infty) \times [0,\infty) \to \mathbb{R}$ by
\begin{equation} \label{eq:RK.cal.R}
\RK_{\mathcal{R}}(r,\rho) \DEF \int_0^\infty \indicatorfunction_{[R,\infty)}(r) \indicatorfunction_{[R, \infty)}(\rho) \, \phi(R) \dd R, \qquad r, \rho \geq 0.
\end{equation}
It can be verified that the function $\RK_{\mathcal{R}}$ is symmetric and positive definite; i.e., for all $a_1,\ldots, a_N \in \mathbb{C}$ and all $x_1,\ldots, x_N \in [0,\infty)$ we have
\begin{equation*}
\sum_{i=1}^N \sum_{j=1}^N a_i \, \RK_{\mathcal{R}}(x_i, x_j) \, \overline{a_j} \geq 0.
\end{equation*}
By \cite{Ar1950} it follows that $\RK_{\mathcal{R}}$ is a reproducing kernel which uniquely defines a reproducing kernel Hilbert space $\mathcal{H}(\RK_{\mathcal{R}})$ with inner product $(\cdot, \cdot)_{\RK_{\mathcal{R}}}$. Set
\begin{equation} \label{eq:Phi}
\Phi(r) \DEF \int_{0}^r \phi(R) \dd R, \qquad r \geq 0.
\end{equation}
Observe that $\Phi(\infty) = \int_0^\infty \phi(R) \dd R = 1$. It can be readily seen that
\begin{equation} \label{eq:RK.cal.R.closed.form}
\RK_{\mathcal{R}}(r,\rho) = \Phi(\min\{r,\rho\}), \qquad r, \rho \geq 0.
\end{equation}
Note that $\RK_{\mathcal{R}}(r,0) = \RK_{\mathcal{R}}(0, \rho) = 0$ for all $r, \rho \in [0,\infty)$.

Further, for $\PT{x}^*, \PT{y}^* \in \mathbb{S}^d$ let the kernel $\RK_{\mathcal{S}}$ be defined by
\begin{equation*}
\RK_{\mathcal{S}}( \PT{x}^*, \PT{y}^* ) \DEF \int_{-1}^1 \int_{\mathbb{S}^d} \indicatorfunction_{\mathcal{C}( \PT{x}^*, t )}( \PT{z}^* ) \indicatorfunction_{\mathcal{C}( \PT{y}^*, t )}( \PT{z}^* ) \dd \sigma_d( \PT{z}^* ) \dd t, \qquad \PT{x}^*, \PT{y}^* \in \mathbb{S}^d.
\end{equation*}
The function $\RK_{\mathcal{S}}$ is again symmetric and positive definite and therefore a reproducing kernel which uniquely defines a reproducing kernel Hilbert space $\mathcal{H}(\RK_{\mathcal{S}})$, see \cite{BrDi2013}. The latter also gives the closed form representation
\begin{equation*}
\RK_{\mathcal{S}}( \PT{x}^*, \PT{y}^* ) = 1 - C_d \left\| \PT{x}^* - \PT{y}^* \right\|, \qquad \PT{x}^*, \PT{y}^* \in \mathbb{S}^d,
\end{equation*}
where
\begin{equation} \label{eq:C.d}
C_d \DEF \frac{1}{d} \frac{\omega_{d-1}}{\omega_d} \qquad \text{and} \qquad \frac{\omega_{d-1}}{\omega_d} = \frac{\gammafcn( (d+1)/2 )}{\sqrt{\pi} \, \gammafcn( d/2 )}.
\end{equation}
We remark that
\begin{equation}
\begin{split} \label{eq:W.RK.S}
W( \RK_{\mathcal{S}} )
&\DEF \int_{\mathbb{S}^d} \int_{\mathbb{S}^d} \RK_{\mathcal{S}}( \PT{x}^*, \PT{y}^* ) \dd \sigma_d( \PT{x}^* ) \dd \sigma_d( \PT{y}^* ) = \int_{\mathbb{S}^d} \RK_{\mathcal{S}}( \PT{x}^*, \PT{y}^* ) \dd \sigma_d( \PT{x}^* ) \\
&= 1 - C_d \int_{\mathbb{S}^d} \left\| \PT{x}^* - \PT{y}^* \right\| \dd \sigma_d( \PT{x}^* ) = 1 - C_d \, W( \mathbb{S}^d ), \qquad \PT{y}^* \in \mathbb{S}^d,
\end{split}
\end{equation}
where 
\begin{equation} \label{eq:W.S.d}
W( \mathbb{S}^d ) \DEF \int_{\mathbb{S}^d} \left\| \PT{x}^* - \PT{y}^* \right\| \dd \sigma_d( \PT{x}^* ) = 2^d \frac{\gammafcn((d+1)/2) \gammafcn((d+1)/2)}{\sqrt{\pi} \gammafcn(d + 1/2)}.
\end{equation}

We define now a reproducing kernel Hilbert space on $\mathbb{R}^{d+1}$ with reproducing kernel
\begin{equation} \label{eq:mathcal.K.2nd.form.all}
\begin{split}
\RK(r \PT{x}^*, \rho \PT{y}^*)
&\DEF \RK_{\mathcal{R}}(r, \rho) \, \RK_{\mathcal{S}}( \PT{x}^*, \PT{y}^*) \\
&= \Phi( \min\{ r, \rho \} ) \left( 1 - C_d \left\| \PT{x}^* - \PT{y}^* \right\| \right), \qquad r, \rho \geq 0, \, \PT{x}^*, \PT{y}^* \in \mathbb{S}^d,
\end{split}
\end{equation}
From \cite{Ar1950} we obtain that $\RK$ is a reproducing kernel with corresponding reproducing kernel Hilbert space given by $\mathcal{H}(\RK) = \mathcal{H}(\RK_{\mathcal{R}}) \times \mathcal{H}(\RK_{\mathcal{S}})$, i.e., as the tensor product space of $\mathcal{H}(\RK_{\mathcal{R}})$ and $\mathcal{H}(\RK_{\mathcal{S}})$.
Note that $\RK(\PT{x}, \PT{y}) = 0$ for $\PT{x} = \PT{0}$, which implies that for any function $f \in \mathcal{H}(\RK)$ we have $f(\PT{x}) = 0$, i.e. the functions are anchored at the origin.

Let $\indicatorfunction_{\mathcal{C}( \PT{z}^*, t; R )}$ be the indicator function for  $\mathcal{C}( \PT{z}^*, t; R )$. Since for $\rho \geq 0$ and $\PT{y}^* \in \mathbb{S}^d$
\begin{align*}
\indicatorfunction_{\mathcal{C}( \PT{z}^*, t; R )}( \rho \PT{y}^* )
&= \indicatorfunction_{[R, \infty)}( \rho ) \indicatorfunction_{\mathcal{C}( \PT{z}^*, t )}( \PT{y}^* ) = \indicatorfunction_{[R, \infty)}( \rho ) \indicatorfunction_{[t,1]}( \PT{z}^* \cdot \PT{y}^* ) = \indicatorfunction_{[R,\infty)}( \rho ) \indicatorfunction_{\mathcal{C}( \PT{y}^*, t )}( \PT{z}^* ),
\end{align*}
we have
\begin{equation} \label{eq:mathcal.K.1st.form}
\RK( \PT{x}, \PT{y} ) \DEF \int_0^\infty \int_{-1}^1 \int_{\mathbb{S}^d} \indicatorfunction_{\mathcal{C}( \PT{z}^*, t; R )}( \PT{x} )  \indicatorfunction_{\mathcal{C}( \PT{z}^*, t; R )}( \PT{y} )  \dd \sigma_d( \PT{z}^* ) \dd t \, \phi( R ) \dd R, \quad \PT{x}, \PT{y} \in \mathbb{R}^{d+1}.
\end{equation}
Let $(\cdot, \cdot )_{\RK}$ denote the inner product in the reproducing kernel Hilbert space $\mathcal{H}(\RK)$.

Let us consider functions $U: \mathbb{R}^{d+1} \to \mathbb{R}$ which have an integral representation
\begin{equation} \label{eq:U.normal.form}
U( \PT{x} ) = \int_0^\infty \int_{-1}^1 \int_{\mathbb{S}^d} \indicatorfunction_{\mathcal{C}( \PT{z}^*, t; R )}( \PT{x} ) u( \PT{z}^*, t, R )  \dd \sigma_d( \PT{z}^* ) \dd t \, \phi( R ) \dd R, \qquad \PT{x} \in \mathbb{R}^{d+1},
\end{equation}
where the function $U$ is expressed in terms of a function ${u: \mathbb{S}^d \times [-1,1] \times [0,\infty) \to \mathbb{R}}$ with ${u \in \mathbb{L}_2( \mathbb{S}^d \times [-1,1] \times [0,\infty); \mu_d )}$ and $\dd \mu_d( \PT{z}^*, t, R ) = \dd \sigma_d( \PT{z}^* ) \dd t \, \phi( R ) \dd R$. Every function $\PT{x} \mapsto \RK( \PT{x}, \PT{y} )$, $\PT{y} \in \mathbb{R}^{d+1}$ fixed, is of this type with potential function $( \PT{z}^*, t, R ) \mapsto \indicatorfunction_{\mathcal{C}( \PT{z}^*, t; R )}( \PT{y} )$. The functions of type \eqref{eq:U.normal.form} form a linear function space $\mathcal{U}$ whereon one can define an inner product by means of
\begin{equation} \label{eq:inner.product}
\left( U_1, U_2 \right)_{\RK} \DEF \int_0^\infty \int_{-1}^1 \int_{\mathbb{S}^d} u_1( \PT{z}^*, t, R ) u_2( \PT{z}^*, t, R )  \dd \sigma_d( \PT{z}^* ) \dd t \, \phi( R ) \dd R, \quad U_1, U_2 \in \mathcal{U},
\end{equation}
and a corresponding norm
\begin{equation} \label{eq:norm}
\left\| U \right\|_{\RK} \DEF \left\{ \int_0^\infty \int_{-1}^1 \int_{\mathbb{S}^d} \left| u( \PT{z}^*, t, R ) \right|^2  \dd \sigma_d( \PT{z}^* ) \dd t \, \phi( R ) \dd R \right\}^{1/2}, \qquad U \in \mathcal{U}.
\end{equation}
These definitions yield that for $U \in \mathcal{U}$ with $\| U \|_{\RK} < \infty$,
\begin{equation*}
\left( U, \RK( \PT{\cdot}, \PT{y} ) \right)_{\RK} = \int_0^\infty \int_{-1}^1 \int_{\mathbb{S}^d} u( \PT{z}^*, t, R ) \indicatorfunction_{\mathcal{C}( \PT{z}^*, t; R )}( \PT{y} )  \dd \sigma_d( \PT{z}^* ) \dd t \, \phi( R ) \dd R = U( \PT{y} ), \quad \PT{y} \in \mathbb{R}^{d+1}.
\end{equation*}
We remark that $\RK( \PT{\cdot}, \PT{y} ) \in \mathcal{U}$ for all $\PT{y} \in \mathbb{R}^{d+1}$ and $\| \RK( \PT{\cdot}, \PT{y} ) \|_{\RK}^2 = \RK(\PT{y}, \PT{y}) = \Phi(\|\PT{y}\|) < \infty$.
Hence the uniqueness properties of the reproducing kernel and inner product and norm defined by this kernel imply that all $U \in \mathcal{U}$ with $\| U \|_{\RK} < \infty$ are also in $\mathcal{H}(\RK)$ and the inner product of such functions in $\mathcal{H}(\RK)$ can be written as \eqref{eq:inner.product}.
The reproducing kernel Hilbert space $\mathcal{H}(\RK)$ is then the completion of $\{ U \in \mathcal{U} : ( U, U )_{\RK} < \infty \}$ with respect to \eqref{eq:inner.product}. In fact, we show in Appendix~\ref{appdx:A} that every $f \in \mathcal{H}(\RK)$ has an integral representation \eqref{eq:U.normal.form}.

We make now the following observation, namely, that the reproducing kernel Hilbert spaces $\mathcal{H}(\RK_{\mathcal{S}})$ defined on the sphere $\mathbb{S}^d$ and $\mathcal{H}(\RK_{\mathcal{R}})$ defined on $[0,\infty)$ are isomorphically embedded in $\mathcal{H}(\RK)$. Indeed, let $f \in \mathcal{H}(\RK)$ be the potential function of $u(\PT{z}^*, t, R)$.

First, assume that $u(\PT{z}^*, t, R) = u(\PT{z}^*, t)$ for all $R \geq 0$; that is, $u$ is independent of $R$. Then, for $r \geq 0$ and $\PT{x}^* \in \mathbb{S}^d$ we have
\begin{equation*}
f( r \PT{x}^* ) = \left[ \int_0^\infty \indicatorfunction_{[R,\infty)}( r ) \phi(R) \dd R \right] \left[ \int_{-1}^1 \int_{\mathbb{S}^d} \indicatorfunction_{\mathcal{C}( \PT{z}^*, t)}( \PT{x}^* ) u( \PT{z}^*, t ) \dd \sigma_d( \PT{z}^* ) \dd t \right] = \Phi( r ) \, g( \PT{x}^* ),
\end{equation*}
where the function $g$, given by the second square-bracketed expression, is in $\mathcal{H}( \RK_{\mathcal{S}} )$ as $g$ is the potential function of $u( \PT{z}^*, t)$ (cf. \cite{BrDi2013}). Hence, by \eqref{eq:norm},
\begin{equation} \label{em1}
\left\| f \right\|_{\mathcal{H}(\RK)}^2 = \left[ \int_0^\infty \phi( R ) \dd R \right] \left[ \int_{-1}^1 \int_{\mathbb{S}^d} \left| u( \PT{z}^*, t ) \right|^2 \dd \sigma_d( \PT{z}^* ) \dd t \right] = 1 \times \left\| g \right\|_{\mathcal{H}(\RK_{\mathcal{S}})}^2.
\end{equation}

On the other hand, assume now that $u(\PT{z}^*, t, R) = u( R)$ for all $\PT{z}^* \in \mathbb{S}^d$ and ${t \in [-1,1]}$. Then, for $r \geq 0$ and $\PT{x}^* \in \mathbb{S}^d$ we have
\begin{equation*}
f( r \PT{x}^* ) = \left[ \int_0^\infty \indicatorfunction_{[R,\infty)}( r ) u(R) \phi(R) \dd R \right] \left[ \int_{-1}^1 \int_{\mathbb{S}^d} \indicatorfunction_{\mathcal{C}( \PT{z}^*, t)}( \PT{x}^* ) \dd \sigma_d( \PT{z}^* ) \dd t \right] =  F( r ),
\end{equation*}
where the second square-bracketed expression evaluates to $1$ and the function $F$, given by the first square-bracketed expression, is in $\mathcal{H}( \RK_{\mathcal{R}} )$ as $F$ is the potential function of $u( R )$ (the last statement follows by the same arguments as used in Appendix~\ref{appdx:A}). Again, by \eqref{eq:norm},
\begin{equation} \label{em2}
\left\| f \right\|_{\mathcal{H}(\RK)}^2 = \left[ \int_0^\infty \left| u( R ) \right|^2 \phi( R ) \dd R \right] \left[ \int_{-1}^1 \int_{\mathbb{S}^d} \dd \sigma_d( \PT{z}^* ) \dd t \right] = \left\| F \right\|_{\mathcal{H}(\RK_{\mathcal{R}})}^2 \times 2.
\end{equation}


\subsection{Worst-case error}
\label{subsec:wce}

The \emph{worst-case error} of a QMC method
\begin{equation*}
\numint[X_N](f) = \frac{1}{N} \sum_{j=1}^N f( \PT{x}_j )
\end{equation*}
with node set $X_N = \{ \PT{x}_1, \dots, \PT{x}_N \}$ in $\mathbb{R}^{d+1}$ approximating the integral
\begin{equation*}
\xctint[\psi](f) = \int_{\mathbb{R}^{d+1}} f( \PT{x} ) \, \psi( \PT{x} ) \dd \lambda_{d+1}( \PT{x} )
\end{equation*}
with respect to the probability density function $\psi: \mathbb{R}^{d+1} \to \mathbb{R}$ and the Lebesgue measure $\lambda_{d+1}$ on $\mathbb{R}^{d+1}$ for functions from the unit ball in the reproducing kernel Hilbert space $\mathcal{H}(\RK) \DEF \mathcal{H}( \mathbb{R}^{d+1}; \RK)$ defined by the kernel \eqref{eq:mathcal.K.1st.form} is given by
\begin{equation*}
\wce( \numint[X_N]; \mathcal{H}(\RK) ) \DEF \sup \Big\{ \left| \numint[X_N](f) - \xctint[\psi](f) \right| : f \in \mathcal{H}(\RK), \left\| f \right\|_{\RK} \leq 1 \Big\}.
\end{equation*}

Using the integral representation \eqref{eq:mathcal.K.1st.form}, we have that
\begin{equation*}
\frac{1}{N} \sum_{j=1}^N \RK( \PT{x}, \PT{x}_j ) = \int_0^\infty \int_{-1}^1 \int_{\mathbb{S}^d} \indicatorfunction_{\mathcal{C}( \PT{z}^*, t; R )}( \PT{x} ) \left[ \frac{1}{N} \sum_{j=1}^N \indicatorfunction_{\mathcal{C}( \PT{z}^*, t; R )}( \PT{x}_j ) \right] \dd \sigma_d( \PT{z}^* ) \dd t \, \phi( R ) \dd R
\end{equation*}
and
\begin{equation*}
\begin{split}
&\int_{\mathbb{R}^{d+1}} \RK( \PT{x}, \PT{y} ) \, \psi( \PT{y} ) \dd \lambda_{d+1}( \PT{y} ) \\
&\phantom{equals}= \int_0^\infty \int_{-1}^1 \int_{\mathbb{S}^d} \indicatorfunction_{\mathcal{C}( \PT{z}^*, t; R )}( \PT{x} ) \left[ \int_{\mathbb{R}^{d+1}} \indicatorfunction_{\mathcal{C}( \PT{z}^*, t; R )}( \PT{y} ) \, \psi( \PT{y} ) \dd \lambda_{d+1}( \PT{y} ) \right] \dd \sigma_d( \PT{z}^* ) \dd t \, \phi( R ) \dd R.
\end{split}
\end{equation*}
Thus, the ``representer'' of the numerical integration error of $\numint[X_N](f)$ for $f \in \mathcal{H}(\RK)$,
\begin{subequations}
\begin{align}
\mathcal{R}[X_N]( \PT{x} )
&\DEF \frac{1}{N} \sum_{j=1}^N \RK( \PT{x}, \PT{x}_j ) - \int_{\mathbb{R}^{d+1}} \RK( \PT{x}, \PT{y} ) \, \psi( \PT{y} ) \dd \lambda_{d+1}( \PT{y} ) \label{eq:representer.A} \\
&= \int_0^\infty \int_{-1}^1 \int_{\mathbb{S}^d} \indicatorfunction_{\mathcal{C}( \PT{z}^*, t; R )}( \PT{x} ) \, \delta[X_N]( \PT{z}^*, t, R ) \dd \sigma_d( \PT{z}^* ) \dd t \, \phi( R ) \dd R, \label{eq:representer.B}
\end{align}
\end{subequations}
is of the form \eqref{eq:U.normal.form}, where the potential function is the \emph{local discrepancy function}
\begin{equation} \label{eq:local.discrepancy.function}
\delta[X_N]( \PT{z}^*, t, R ) \DEF \frac{1}{N} \sum_{j=1}^N \indicatorfunction_{\mathcal{C}( \PT{z}^*, t; R )}( \PT{x}_j ) - \int_{\mathbb{R}^{d+1}} \indicatorfunction_{\mathcal{C}( \PT{z}^*, t; R )}( \PT{y} ) \, \psi( \PT{y} ) \dd \lambda_{d+1}( \PT{y} ).
\end{equation}
The name ``representer'' of the numerical integration error is justified because of
\begin{equation}
\numint[X_N](f) - \xctint[\psi](f) = \left( f, \mathcal{R}[X_N] \right)_{\RK}.
\end{equation}
Application of the Cauchy-Schwarz inequality gives the Koksma-Hlawka like inequality
\begin{equation*}
\left| \numint[X_N](f) - \xctint[\psi](f) \right| \leq \left\| f \right\|_{\RK} \left\| \mathcal{R}[X_N] \right\|_{\RK}, \qquad f \in \mathcal{H}(\RK).
\end{equation*}
Equality is assumed for $f \equiv \mathcal{R}[X_N]$; cf., e.g., \cite[Ch.~2]{DiPi2010} and \cite{Hi1998}.

\begin{subequations}
Utilizing \eqref{eq:inner.product} and \eqref{eq:representer.A} and the reproducing property of $\RK$, we obtain
\begin{align}
\left[ \wce( \numint[X_N]; \mathcal{H}(\RK) ) \right]^2
&= \left( \mathcal{R}[X_N], \mathcal{R}[X_N] \right)_{\RK} \notag \\
\begin{split} \label{eq:wce.form.a}
&= \frac{1}{N^2} \mathop{\sum_{i=1}^N \sum_{j=1}^N} \RK( \PT{x}_i, \PT{x}_j ) - \frac{2}{N} \sum_{j=1}^N \int_{\mathbb{R}^{d+1}} \RK( \PT{x}, \PT{x}_j ) \, \psi( \PT{x} ) \dd \lambda_{d+1}( \PT{x} ) \\
&\phantom{=}+ \int_{\mathbb{R}^{d+1}} \int_{\mathbb{R}^{d+1}} \RK( \PT{x}, \PT{y} ) \, \psi( \PT{x} ) \psi( \PT{y} ) \dd \lambda_{d+1}( \PT{x} ) \dd \lambda_{d+1}( \PT{y} ).
\end{split}
\end{align}
Utilizing \eqref{eq:norm} and \eqref{eq:representer.B}, we obtain the ``discrepancy form'' of the squared worst-case error
\begin{equation} \label{eq:wce.form.b}
\begin{split}
\left[ \wce( \numint[X_N]; \mathcal{H}(\RK) ) \right]^2
&= \left\| \mathcal{R}[X_N] \right\|_{\RK}^2 \\
&= \int_0^\infty \int_{-1}^1 \int_{\mathbb{S}^d} \left| \delta[X_N]( \PT{z}^*, t, R ) \right|^2  \dd \sigma_d( \PT{z}^* ) \dd t \, \phi( R ) \dd R.
\end{split}
\end{equation}
\end{subequations}
The last result motivates the definition of the \emph{spherical cone $\mathbb{L}_2$-discrepancy} of an $N$-point configuration $X_N = \{ \PT{x}_1, \dots, \PT{x}_N \} \subseteq \mathbb{R}^{d+1}$,
\begin{equation} \label{eq:L.2.discrepancy}
D_{\mathbb{L}_2}^{\mathrm{SC}}( X_N ) \DEF \left( \int_0^\infty \int_{-1}^1 \int_{\mathbb{S}^d} \left| \delta[X_N]( \PT{z}^*, t, R ) \right|^2  \dd \sigma_d( \PT{z}^* ) \dd t \, \phi( R ) \dd R \right)^{1/2}.
\end{equation}
The right-hand side of \eqref{eq:wce.form.a} does not change when a constant is added to the kernel $\RK$. This enables us to write the worst-case error formula in a more compact way, 
\begin{equation}
\left[ \wce( \numint[X_N]; \mathcal{H}(\RK) ) \right]^2 = \frac{1}{N^2} \mathop{\sum_{i=1}^N \sum_{j=1}^N} \RKmod( \PT{x}_i, \PT{x}_j ) - \frac{2}{N} \sum_{j=1}^N \int_{\mathbb{R}^{d+1}} \RKmod( \PT{x}, \PT{x}_j ) \, \psi( \PT{x} ) \dd \lambda_{d+1}( \PT{x} ),
\end{equation}
where $\RKmod: \mathbb{R}^{d+1} \times \mathbb{R}^{d+1} \to \mathbb{R}$ is defined by
\begin{equation} \label{eq:reproducing.kernel.modified}
\RKmod( \PT{x}, \PT{y} ) \DEF \RK( \PT{x}, \PT{y} ) - W(\RK), \qquad \PT{x}, \PT{y} \in \mathbb{R}^{d+1},
\end{equation}
and
\begin{equation} \label{eq:W.reproducing.kernel.K}
W(\RK) \DEF \int_{\mathbb{R}^{d+1}} \int_{\mathbb{R}^{d+1}} \RK( \PT{x}, \PT{y} ) \, \psi( \PT{x} ) \psi( \PT{y} ) \dd \lambda_{d+1}( \PT{x} ) \dd \lambda_{d+1}( \PT{y} ).
\end{equation}
In the following the use of the calligraphic symbol for the kernel $\RK$ is reserved to indicate the subtraction of the constant $W(\RK)$ from $\RK$.

We summarize these observations in the following theorem.

\begin{thm} \label{thm:wce.form.A}
Let $\mathcal{H}(\RK)$ be the Hilbert space uniquely defined by the reproducing kernel \eqref{eq:mathcal.K.1st.form} with closed form \eqref{eq:mathcal.K.2nd.form.all}. Then for a method $\numint[X_N]$ with node set $X_N = \{ \PT{x}_1, \dots, \PT{x}_N \} \subseteq \mathbb{R}^{d+1}$,
\begin{align*}
\wce( \numint[X_N]; \mathcal{H}(\RK) )
&= \Bigg( \frac{1}{N^2} \mathop{\sum_{i=1}^N \sum_{j=1}^N} \RKmod( \PT{x}_i, \PT{x}_j ) - \frac{2}{N} \sum_{j=1}^N \int_{\mathbb{R}^{d+1}} \RKmod( \PT{x}, \PT{x}_j ) \, \psi( \PT{x} ) \dd \lambda_{d+1}( \PT{x} ) \Bigg)^{1/2} \\
&= D_{\mathbb{L}_2}^{\mathrm{SC}}( X_N ),
\end{align*}
where the spherical cone $\mathbb{L}_2$-discrepancy is defined in \eqref{eq:L.2.discrepancy}.
\end{thm}

From \eqref{eq:wce.form.a} it follows that an $N$-point configuration $X_N^*$ that minimizes
\begin{equation} \label{eq:external.field.problem}
\frac{1}{N^2} \mathop{\sum_{i=1}^N \sum_{j=1}^N} \RK( \PT{x}_i, \PT{x}_j ) - \frac{2}{N} \sum_{j=1}^N \int_{\mathbb{R}^{d+1}} \RK( \PT{x}, \PT{x}_j ) \, \psi( \PT{x} ) \dd \lambda_{d+1}( \PT{x} )
\end{equation}
has smallest worst-case error $\wce( \numint[X_N^*]; \mathcal{H}(\RK) )$ and, by \eqref{eq:wce.form.b}, smallest \emph{spherical cone $\mathbb{L}_2$-discrepancy} $D_{\mathbb{L}_2}^{\mathrm{SC}}( X_N^* )$. 
The kernel $\RK( \PT{x}, \PT{y} )$ has the representation \eqref{eq:mathcal.K.2nd.form.all}.
The expression \eqref{eq:external.field.problem} can be interpreted as the ``energy'' of the nodes $\PT{x}_1, \dots, \PT{x}_N$ subject to an external field
\begin{equation*}
\mathcal{Q}( \PT{y} ) \DEF - \int_{\mathbb{R}^{d+1}} \RK( \PT{x}, \PT{y} ) \, \psi( \PT{x} ) \dd \lambda_{d+1}( \PT{x} ), \qquad \PT{y} \in \mathbb{R}^{d+1},
\end{equation*}
which prevents the nodes from escaping to infinity. (Indeed, by definition of the kernel $\RK$ (see \eqref{eq:mathcal.K.2nd.form.all}) the contribution to \eqref{eq:external.field.problem} (``point energy'') of a point $\PT{x}_{j_0}$ tends to $0$ as $\|\PT{x}_{j_0}\| \to \infty$. On the other hand, the worst-case error goes to $0$ as $N \to \infty$ only if the energy \eqref{eq:external.field.problem} becomes negative in order to compensate the positive double integral in \eqref{eq:wce.form.a}.)


A standard probabilistic argument yields the following result for the root mean square error of a QMC method for a \emph{typical} $N$-point node set.
\begin{thm} \label{thm:rms.wce.form.A}
Let $\mathcal{H}(\RK)$ be the Hilbert space uniquely defined by the reproducing kernel \eqref{eq:mathcal.K.1st.form} with closed form \eqref{eq:mathcal.K.2nd.form.all}. Then
\begin{equation*}
\sqrt{\mathbb{E}\big[ \{ \wce( \numint[\{ \PT{y}_1, \dots, \PT{y}_N \}]; \mathcal{H}(\RK) ) \}^2 \big]} = \frac{1}{\sqrt{N}} \left( \int_{\mathbb{R}^{d+1}} \Phi( \| \PT{x} \| ) \, \psi( \PT{x} ) \dd \lambda_{d+1}( \PT{x} ) - W( \RK ) \right)^{1/2},
\end{equation*}
where the points $\PT{y}_1, \dots, \PT{y}_N$ are independently and identically $\psi \lambda_{d+1}$-distributed in $\mathbb{R}^{d+1}$.
\end{thm}


\begin{proof}
Consider the product probability measure
\begin{equation*}
\eta( \PT{y}_1, \PT{y}_2, \dots, \PT{y}_N) \DEF \psi( \PT{y}_1 ) \lambda_{d+1}( \PT{y}_1 ) \psi( \PT{y}_2 ) \lambda_{d+1}( \PT{y}_2 ) \cdots \psi( \PT{y}_N ) \lambda_{d+1}( \PT{y}_N ).
\end{equation*}
By Theorem~\ref{thm:wce.form.A}, the expected value of the squared worst-case error is
\begin{align*}
&\mathbb{E}\big[ \{ \wce( \numint[\{ \PT{y}_1, \dots, \PT{y}_N \}]; \mathcal{H}(\RK) ) \}^2 \big] \\
&\phantom{equals}= \mathop{\int \cdots \int}_{\mathbb{R}^{d+1} \times \cdots \times \mathbb{R}^{d+1}} \{ \wce( \numint[\{ \PT{y}_1, \dots, \PT{y}_N \}]; \mathcal{H}(\RK) ) \}^2 \dd \eta( \PT{y}_1, \PT{y}_2, \dots, \PT{y}_N) \\
&\phantom{equals}= \mathop{\int \cdots \int}_{\mathbb{R}^{d+1} \times \cdots \times \mathbb{R}^{d+1}} \Bigg[ \frac{1}{N^2} \sum_{j=1}^N \RKmod( \PT{y}_j, \PT{y}_j ) + \frac{1}{N^2} \mathop{\sum_{i=1}^N \sum_{j=1}^N}_{i \neq j} \RKmod( \PT{y}_i, \PT{y}_j ) \Bigg] \dd \eta( \PT{y}_1, \PT{y}_2, \dots, \PT{y}_N) \\
&\phantom{equals=}- \frac{2}{N} \sum_{j=1}^N \mathop{\int \cdots \int}_{\mathbb{R}^{d+1} \times \cdots \times \mathbb{R}^{d+1}} \int_{\mathbb{R}^{d+1}} \RKmod( \PT{x}, \PT{y}_j ) \, \psi( \PT{x} ) \dd \lambda_{d+1}( \PT{x} ) \dd \eta( \PT{y}_1, \PT{y}_2, \dots, \PT{y}_N).
\end{align*}
The measure $\eta$ is the product of the probability measures $\psi( \PT{y}_j ) \lambda_{d+1}( \PT{y}_j )$, $1 \leq j \leq N$. Hence
\begin{align*}
&\mathbb{E}\big[ \{ \wce( \numint[\{ \PT{y}_1, \dots, \PT{y}_N \}]; \mathcal{H}(\RK) ) \}^2 \big] = \frac{1}{N} \int_{\mathbb{R}^{d+1}} \RKmod( \PT{y}, \PT{y} ) \psi( \PT{y} ) \dd \lambda_{d+1}( \PT{y} ) \\
&\phantom{equals=}+ \left[ \frac{N \left( N - 1 \right)}{N^2} - \frac{2N}{N} \right] \int_{\mathbb{R}^{d+1}} \int_{\mathbb{R}^{d+1}} \RKmod( \PT{x}, \PT{y} ) \, \psi( \PT{x} ) \psi( \PT{y} ) \dd \lambda_{d+1}( \PT{x} ) \dd \lambda_{d+1}( \PT{y} ).
\end{align*}
The double integral above is zero by definition of $\RKmod$, cf. \eqref{eq:reproducing.kernel.modified}. Therefore
\begin{equation*}
\mathbb{E}\big[ \{ \wce( \numint[\{ \PT{y}_1, \dots, \PT{y}_N \}]; \mathcal{H}(\RK) ) \}^2 \big] = \frac{1}{N} \int_{\mathbb{R}^{d+1}} \RKmod( \PT{y}, \PT{y} ) \psi( \PT{y} ) \dd \lambda_{d+1}( \PT{y} ).
\end{equation*}
The result follows from (see \eqref{eq:mathcal.K.2nd.form.all} and \eqref{eq:reproducing.kernel.modified}) $\RKmod( \PT{y}, \PT{y} ) = \Phi( \| \PT{y} \| ) - W(\RK)$ for $\| \PT{y} \| \geq 0$.
\end{proof}

\subsection{A lower bound for the worst-case error}

For our Hilbert spaces $\mathcal{H} = \mathcal{H}( \RK )$, $\mathcal{H}( \RK_{\mathcal{R}} )$, $\mathcal{H}( \RK_{\mathcal{S}} )$ let
$\minwce(\mathcal{H}; N)$ denote the infimum of the worst-case error of numerical integration when extended over all integration algorithms that use $N$ function evaluations.

Recall that the reproducing kernel Hilbert spaces $\mathcal{H}(\RK_{\mathcal{S}})$ and $\mathcal{H}(\RK_{\mathcal{R}})$ are both isomorphically embedded in $\mathcal{H}(\RK)$ with the constants given in \eqref{em1} and \eqref{em2}. This implies that
\begin{equation*}
\minwce(\mathcal{H}(\RK); N) \geq \max\{ \minwce(\mathcal{H}(\RK_{\mathcal{S}}); N), \sqrt{2} \minwce(\mathcal{H}(\RK_{\mathcal{R}}); N) \}.
\end{equation*}

The Hilbert space $\mathcal{H}(\RK_{\mathcal{S}})$ can be identified with a certain Sobolev space of smoothness $s=(d+1)/2$ (cf. \cite{BrDi2013}) and for such spaces \cite{He2006,HeSl2005b} obtained optimal lower bounds for the worst-case error of order $N^{-s/d}$; i.e., 
%
there is a constant $C > 0$ independent of $N$ such that
\begin{equation*}
\minwce(\mathcal{H}(\RK_{\mathcal{S}}); N) \ge C N^{-1/2 - 1/(2d)} \quad \mbox{for all } N \ge 1.
\end{equation*}
Thus we obtain the following theorem.
\begin{thm} \label{thm:wce.lower.bound}
There is a constant $c_d > 0$ which depends only on $d$, such that the minimal worst-case error for integration in $\mathcal{H}(\RK)$ is bounded by
\begin{equation*}
\minwce(\mathcal{H}(\RK); N) \geq c_d N^{-1/2 - 1/(2d)} \quad \mbox{for all } N \ge 1.
\end{equation*}
\end{thm}

\subsection{Isotropic weight (or density) function $\psi( \PT{x} )$}

From here on we assume that the probability density function $\psi( \PT{x} )$ in the exact integral $\xctint[\psi]$ is isotropic; i.e., a radial function
\begin{equation} \label{eq:psi.isotropic}
\psi( \PT{x} ) = h( \| \PT{x} \| ), \qquad \PT{x} \in \mathbb{R}^{d+1},
\end{equation}
for some function $h:[0,\infty) \to [0,\infty)$, such that after a change to spherical coordinates,
\begin{equation} \label{eq:psi.isotropic.normalization}
1 = \int_{\mathbb{R}^{d+1}} \psi( \PT{x} ) \, \dd \lambda_{d+1}( \PT{x} ) = \int_0^{\infty} \psidensity( r ) \, \dd r.
\end{equation}
Examples of such probability density functions will be considered in Section~\ref{sec:examples.A}.
%
We define
\begin{equation} \label{eq:psi.isotropic.density.and.cdf}
\psidensity(r) \DEF \omega_d \, h(r) \, r^d, \quad r \geq 0, \qquad \Psi( \rho ) \DEF \int_0^\rho \psidensity(r) \dd r, \quad \rho \geq 0,
\end{equation}
and
\begin{equation}
W( \RK_{\mathcal{R}}, \psidensity ) \DEF \int_0^\infty \int_0^\infty \Phi( \min\{ r, \rho \} ) \, \psidensity( r )  \psidensity( \rho ) \dd r \dd \rho.
\end{equation}

\begin{thm} \label{thm:wce.form.B}
Let $\mathcal{H}( \RK )$ be the Hilbert space uniquely defined by the reproducing kernel \eqref{eq:mathcal.K.1st.form} with closed form \eqref{eq:mathcal.K.2nd.form.all} and the density $\psi$ be isotropic satisfying \eqref{eq:psi.isotropic} and \eqref{eq:psi.isotropic.normalization}. Suppose \eqref{eq:psi.isotropic.density.and.cdf}. Then for a method $\numint[X_N]$ with node set $X_N = \{ \PT{x}_1, \dots, \PT{x}_N \} \subseteq \mathbb{R}^{d+1}$,
\begin{equation*}
\begin{split}
&\wce( \numint[X_N]; \mathcal{H}(\RK) )
= \Bigg( \frac{1}{N^2} \mathop{\sum_{i=1}^N \sum_{j=1}^N} \RKmod( \PT{x}_i, \PT{x}_j ) \\
&\phantom{equals=}- \frac{2}{N} \sum_{j=1}^N \Bigg[ \Phi( \| \PT{x}_j \| ) \, \int_{\| \PT{x}_j \|}^\infty \psidensity( r ) \dd r + \int_0^{\| \PT{x}_j \|}  \Phi( r ) \psidensity( r ) \dd r -  W( \RK_{\mathcal{R}}, \psidensity ) \Bigg] W( \RK_{\mathcal{S}} ) \Bigg)^{1/2}. 
\end{split}
\end{equation*}
\end{thm}

\begin{proof}
Let $\rho = \| \PT{y} \|$.
A change to spherical coordinates, \eqref{eq:mathcal.K.2nd.form.all}, \eqref{eq:W.RK.S} and \eqref{eq:reproducing.kernel.modified} yields
\begin{align}
\int_{\mathbb{R}^{d+1}} \RK( \PT{x}, \PT{y} ) \, \psi( \PT{x} ) \, \dd \lambda_{d+1}( \PT{x} )
&= \int_0^\infty \int_{\mathbb{S}^d} \RK_{\mathcal{S}}( \PT{x}^*, \PT{y}^* ) \Phi( \min\{ r, \rho \} ) \dd \sigma_d( \PT{x}^* ) \psidensity( r ) \dd r \notag \\
&= \left[ \int_0^\infty \Phi( \min\{ r, \rho \} ) \, \psidensity( r )  \dd r \right] \left[ \int_{\mathbb{S}^d} \RK_{\mathcal{S}}( \PT{x}^*, \PT{y}^* ) \dd \sigma_d( \PT{x}^* ) \right] \notag \\
&= \left[ \Phi( \rho ) \,\int_\rho^\infty \psidensity( r ) \dd r + \int_0^\rho \Phi( r ) \psidensity( r ) \dd r \right] W( \RK_{\mathcal{S}} ) \label{eq:single.integral.of.RK}
\end{align}
%
and (recall \eqref{eq:W.reproducing.kernel.K})
\begin{align}
W(\RK)
&= \left[ \int_0^\infty \int_0^\infty \Phi( \min\{ r, \rho \} ) \, \psidensity( r )  \psidensity( \rho ) \dd r \dd \rho \right] \left[ \int_{\mathbb{S}^d} \RK_{\mathcal{S}}( \PT{x}^*, \PT{y}^* ) \dd \sigma_d( \PT{x}^* ) \right] \notag \\
&= W( \RK_{\mathcal{R}}, \psidensity ) W( \RK_{\mathcal{S}} ). \label{eq:W.K.factorization}
\end{align}
Substitution into the worst-case error formula in Theorem~\ref{thm:wce.form.A} gives the desired result.
\end{proof}

It can be easily seen that
\begin{align}
W( \RK_{\mathcal{R}}, \psidensity )
&= 2 \int_0^\infty \Phi(\rho) \left\{ \int_\rho^\infty \psidensity( r ) \dd r \right\} \psidensity( \rho ) \dd\rho \label{eq:double.int.isotropic} \\
&= 1 - 2 \int_0^\infty \left( 1 - \Phi(\rho) \right) \left\{ \int_\rho^\infty \psidensity( r ) \dd r \right\} \psidensity( \rho ) \dd\rho. \label{eq:double.int.isotropic.B}
\end{align}

For further references we record the following consequence of the proof of Theorem~\ref{thm:wce.form.B},
\begin{equation} \label{eq:mean.zero.identity}
\int_{\mathbb{R}^{d+1}} \Bigg[ \Phi( \| \PT{y} \| ) \, \int_{\| \PT{y} \|}^\infty \psidensity( r ) \dd r + \int_0^{\| \PT{y} \|}  \Phi( r ) \psidensity( r ) \dd r -  W( \RK_{\mathcal{R}}, \psidensity ) \Bigg] \psi( \PT{y} ) \dd \lambda_{d+1}( \PT{y} ) = 0.
\end{equation}



Following our strategy to have a pre-scribed point set $\{ \PT{y}_1^*, \dots, \PT{y}_N^* \} \subseteq \mathbb{S}^d$, we find suitable radii $\rho_1, \dots, \rho_N$ by choosing them at random. The appropiate probability model is imposed by the (radial) probability density function $\psi$. A ``typical'' $N$-point set $\{ \rho_1 \PT{y}_1^*, \dots, \rho_N \PT{y}_N^* \}$ in $\mathbb{R}^{d+1}$ obeying this model will have a worst-case error as follows.

\begin{thm} \label{thm:rms.wce.form.B}
Under the assumptions of Theorem~\ref{thm:wce.form.B},
\begin{equation}
\begin{split} \label{eq:rms.wce.form.B}
&\mathbb{E}\big[ \{ \wce( \numint[\{ \rho_1 \PT{y}_1^*, \dots, \rho_N \PT{y}_N^* \} ]; \mathcal{H}(\RK) ) \}^2 \big] \\
&\phantom{equ}= \frac{1}{N} \left[ \int_0^\infty \Phi( \rho ) \psidensity( \rho ) \dd \rho - W( \RK_{\mathcal{R}}, \psidensity ) \right] + W( \RK_{\mathcal{R}}, \psidensity ) \left[ \frac{1}{N^2} \mathop{\sum_{i=1}^N \sum_{j=1}^N} \RKmod_{\mathcal{S}}( \PT{y}_i^*, \PT{y}_j^* ) \right].
\end{split}
\end{equation}
where $\PT{y}_1^*, \dots, \PT{y}_N^* \in \mathbb{S}^{d}$ are fixed and the radii $\rho_1, \dots, \rho_N$ are independently and identically $\psidensity \lambda_{1}$-distributed.
\end{thm}

\begin{proof}
Let $\PT{y}_1^*, \dots, \PT{y}_N^* \in \mathbb{S}^{d}$ be fixed.
By assumption, the product measure
\begin{equation*}
\eta(\rho_1, \dots, \rho_N) \DEF \psidensity(\rho_1) \lambda_1(\rho_1) \psidensity(\rho_2) \lambda_1(\rho_2) \cdots \psidensity(\rho_N) \lambda_1(\rho_N),
\end{equation*}
formed by the probability measure $\psidensity(\rho) \lambda_1(\rho)$ supported on the interval $(0,\infty)$ is itself a probability measure. Hence, by Theorem~\ref{thm:wce.form.B},
\begin{align*}
&\mathbb{E}\big[ \{ \wce( \numint[\{ \rho_1 \PT{y}_1^*, \dots, \rho_N \PT{y}_N^* \}]; \mathcal{H}(\RK) ) \}^2 \big] \\
&\phantom{equals}= \mathop{\int \cdots \int}_{\mathbb{R} \times \cdots \times \mathbb{R}} \{ \wce( \numint[\{ \rho_1 \PT{y}_1^*, \dots, \rho_N \PT{y}_N^* \}]; \mathcal{H}(\RK) ) \}^2 \dd \eta( \rho_1, \dots, \rho_N) \\
&\phantom{equals}= \mathop{\int \cdots \int}_{\mathbb{R} \times \cdots \times \mathbb{R}} \Bigg[ \frac{1}{N^2} \sum_{j=1}^N \RKmod( \rho_j \PT{y}_j^*, \rho_j \PT{y}_j^* ) + \frac{1}{N^2} \mathop{\sum_{i=1}^N \sum_{j=1}^N}_{i \neq j} \RKmod( \rho_i \PT{y}_i^*, \rho_j \PT{y}_j^* ) \Bigg] \dd \eta( \rho_1, \dots, \rho_N).
\end{align*}
Note that the integral over the single sum in the worst-case error formula vanishes by \eqref{eq:mean.zero.identity}.
Since (using $\RKmod( \PT{y}, \PT{y} ) = \Phi( \| \PT{y} \| ) - W( \RK )$ by \eqref{eq:mathcal.K.2nd.form.all} and \eqref{eq:reproducing.kernel.modified})
\begin{align*}
&\mathop{\int \cdots \int}_{\mathbb{R} \times \cdots \times \mathbb{R}} \RKmod( \rho_j \PT{y}_j^*, \rho_j \PT{y}_j^* ) \dd \eta( \rho_1, \dots, \rho_N) = \int_0^\infty \Phi( \rho ) \psidensity( \rho ) \dd \rho - W( \RK )
\end{align*}
and also (cf. \eqref{eq:W.K.factorization} and \eqref{eq:double.int.isotropic.B})
\begin{align*}
&\mathop{\int \cdots \int}_{\mathbb{R} \times \cdots \times \mathbb{R}} \RKmod( \rho_i \PT{y}_i^*, \rho_j \PT{y}_j^* ) \dd \eta( \rho_1, \dots, \rho_N) \\
&\phantom{equals}= \left[ 1 - C_d \left\| \PT{y}_i^* - \PT{y}_j^* \right\| \right] \int_0^\infty \int_0^\infty \Phi( \min\{ \rho_i, \rho_j \} ) \psidensity( \rho_i )  \psidensity( \rho_j ) \dd \lambda_1(\rho_i) \dd \lambda_1(\rho_j) - W( \RK ) \\
&\phantom{equals}= \left[ 1 - C_d \left\| \PT{y}_i^* - \PT{y}_j^* \right\| \right] W( \RK_{\mathcal{R}}, \psidensity ) - W( \RK_{\mathcal{S}} ) W( \RK_{\mathcal{R}}, \psidensity ) = W( \RK_{\mathcal{R}}, \psidensity ) \RKmod_{\mathcal{S}}( \PT{y}_i^*, \PT{y}_j^* ),
\end{align*} 
it follows that
\begin{align*}
&\mathbb{E}\big[ \{ \wce( \numint[\{ \rho_1 \PT{y}_1^*, \dots, \rho_N \PT{y}_N^* \} ]; \mathcal{H}(\RK) ) \}^2 \big] \\
&\phantom{equals}= \frac{1}{N} \left[ \int_0^\infty \Phi( \rho ) \psidensity( \rho ) \dd \rho - W( \RK ) \right] + W( \RK_{\mathcal{R}}, \psidensity ) \Bigg[ \frac{1}{N^2} \mathop{\sum_{i=1}^N \sum_{j=1}^N}_{i \neq j} \RKmod_{\mathcal{S}}( \PT{y}_i^*, \PT{y}_j^* ) \Bigg].
\end{align*}
Rearrangement of terms gives the desired result.
\end{proof}

The expected value formula in \eqref{eq:rms.wce.form.B} has two components. The first one is related to the randomly chosen radii and is of order $N^{-1}$. The second quantity measures the quality of the $N$-point set $Y_N^* = \{ \PT{y}_1^*, \dots, \PT{y}_N^* \}$. It is the worst-case integration error of the QMC methods with these $N$ nodes for functions in the unit ball in the Sobolev space $\mathbb{H}^{(d+1)/2}( \mathbb{S}^d )$ provided with the reproducing kernel $\RK_{\mathcal{S}}$ (cf. \cite{BrDi2013}). This worst-case error satisfies the relations
\begin{align}
\left[ \wce( \numint[Y_N^*]; \mathcal{H}( \RK_{\mathcal{S}} ) ) \right]^2
&= \frac{1}{N^2} \mathop{\sum_{i=1}^N \sum_{j=1}^N} \RKmod_{\mathcal{S}}( \PT{y}_i^*, \PT{y}_j^* ) = C_d \left[ W( \mathbb{S}^d ) - \frac{1}{N^2} \mathop{\sum_{i=1}^N \sum_{j=1}^N} \left\| \PT{y}_i^* - \PT{y}_j^* \right\| \right] \notag \\
&= \left[ D_{\mathbb{L}_2}^{\mathrm{C}}( Y_N^* ) \right]^2 \DEF \int_{-1}^1 \int_{\mathbb{S}^d} \left| \delta[X_N]( \PT{z}^*, t, 1 ) \right|^2 \dd \sigma_d( \PT{z}^* ) \dd t, \label{eq:wce.Sobolev.sphere}
\end{align}
where $D_{\mathbb{L}_2}^{\mathrm{C}}( Y_N^* )$ is the spherical cap $\mathbb{L}_2$-discrepancy of $Y_N^*$. It is known that $N$-point configurations on $\mathbb{S}^d$ that maximize the sum of all mutual distances (and thus have minimal worst-case error and minimal spherical cap $\mathbb{L}_2$-discrepancy) achieve optimal convergence order $N^{-1/2 - 1/(2d)}$. Such sequences are one example of QMC design sequences for $\mathbb{H}^{(d+1)/2}( \mathbb{S}^d )$ (cf. \cite{BrSaSlWo2013}). So-called \emph{low-discrepancy sequences} on $\mathbb{S}^d$ allow order $\sqrt{\log N} N^{-1/2 - 1/(2d)}$.

Evidently, there is a gap between the order of the lower bound $N^{-1/2-1/(2d)}$ (Theorem~\ref{thm:wce.lower.bound}) and what would be achievable on average by  random selection processes (Theorems~\ref{thm:rms.wce.form.A} and \ref{thm:rms.wce.form.B}). In \cite{BrSaSlWo2013} it is observed that a compartmentalized random selection of points on the sphere improves the decay of the mean square worst-case error. We follow the same stratifying approach here.
Consider the following partition of unity
\begin{align*}
1
&= \int_{\mathbb{R}^{d+1}} \psi( \PT{x} ) \, \dd \lambda_{d+1}( \PT{x} ) = \int_0^\infty \int_{\mathbb{S}^d} \psidensity( r ) \dd \sigma_d \dd r \\
&= \left[ \int_{\mathbb{S}^d} \dd \sigma_d  \right] \left[ \int_0^\infty \psidensity( r ) \dd r \right] = \left[ \sum_{m=1}^M \int_{D_{m,M}} \dd \sigma_d  \right] \left[ \sum_{k=1}^K \int_{\rho_{k-1}}^{\rho_k} \psidensity( r ) \dd r \right].
\end{align*}
We require that $D_{1,M}, \dots, D_{M,M}$ form an equal area partition of $\mathbb{S}^d$ into $M$ subsets and $0 = \rho_0 < \rho_1 < \cdots < \rho_K = \infty$ are such that $\int_{\rho_{k-1}}^{\rho_k} \psidensity( r ) \dd r = 1 / K$ for all $k = 1, \dots, K$. This defines a partition of $\mathbb{R}^{d+1}$ into $N = M K$ parts of equal mass (probability) $1 / N$ given by
\begin{equation*}
A_{m,k}^{(M,K)} \DEF \left\{ \rho \PT{x}^* \in \mathbb{R}^{d+1} : \PT{x}^* \in D_{m,M}, \rho \in (\rho_{k-1}, \rho_k) \right\}, \qquad 1 \leq m \leq M, 1 \leq k \leq K.
\end{equation*}
Such a partition we call \emph{small-diameter equal mass partition} if the sets $D_{1,M}, \dots, D_{M,M}$ satisfy for some positive constant $c$ independent of $j$ and $M$ the small-diameter constraints
\begin{equation*}
\diam D_{j,M} \DEF \sup \big\{ \left\| \PT{x} - \PT{y} \right\| : \PT{x}, \PT{y} \in D_{j,M} \} \leq \frac{c}{M^{1/d}}, \qquad j = 1, \dots, M, M \geq 2;
\end{equation*}
that is, the diameter bound is at the same scale as the well-separation distance of $M$ points on $\mathbb{S}^d$.


\begin{thm} \label{thm:compartmentalized.B}
Under the assumptions of Theorem~\ref{thm:wce.form.B}, let $(A_{m,k}^{(M,K)})$ be a small-diameter equal mass partition of $\mathbb{R}^{d+1}$ into $N = M K$ parts of equal mass $1 / N$.
Then
\begin{equation} \label{eq:compartmentalized.B}
\begin{split}
&\mathbb{E}\Bigg[ \sup_{\substack{f \in \mathcal{H}(\RK), \\ \| f \|_{\RK} \leq 1}} \left| \frac{1}{M K} \sum_{m=1}^M \sum_{k=1}^K f( \PT{y}_{m,k}^{(M,K)} ) - \int_{\mathbb{R}^{d+1}} f( \PT{y} ) \psi( \PT{y} ) \dd \lambda_{d+1}( \PT{y} ) \right|^2 \Bigg] \\
&\phantom{=}= \frac{1}{M K} \left[ \frac{1}{K} \sum_{k=1}^K \left( \int_{\rho_{k-1}}^{\rho_k} \Phi( r ) \frac{\psidensity( r ) \dd r}{1/K} - \int_{\rho_{k-1}}^{\rho_k} \int_{\rho_{k-1}}^{\rho_k} \Phi( \min\{ r, \rho \} ) \frac{\psidensity( r ) \dd r}{1/K} \frac{\psidensity( \rho ) \dd \rho}{1/K} \right) \right] \\
&\phantom{equals=}+ \frac{C_d}{M K} \left[ \frac{1}{K} \sum_{k=1}^K \int_{\rho_{k-1}}^{\rho_k} \int_{\rho_{k-1}}^{\rho_k} \Phi( \min\{ r, \rho \} ) \frac{\psidensity( r ) \dd r}{1/K} \frac{\psidensity( \rho ) \dd \rho}{1/K} \right] \\
&\phantom{equals=\pm \frac{C_d}{M K} }\times \left[ \frac{1}{M} \sum_{m=1}^M \int_{D_{m,M}} \int_{D_{m,M}} \left\| \PT{x}^* - \PT{y}^* \right\| \frac{\dd \sigma_d( \PT{x}^* )}{1/M} \frac{\dd \sigma_d( \PT{y}^* )}{1/M} \right],
\end{split}
\end{equation}
where $\PT{y}_{m,k}^{(M,K)}$ is chosen randomly from $A_{m,k}^{(M,K)}$ with respect to the probability measure $\eta_{m,k}^{(M,K)}$ induced by the density function $\psi$ (i.e., $\dd \eta_{m,k}^{(M,K)}( \rho \PT{y}^* ) = K M \dd \sigma_d|_{D_{m,M}}( \PT{y}^* ) \psidensity( \rho ) \dd \rho |_{[\rho_{k-1}, \rho_k)}$).
\end{thm}

\begin{proof}
Fix  $M$ and $K$. We simplify the notation by dropping the dependence on $M$ and $K$. Let
\begin{equation*}
\dd \eta( \PT{y}_{1,1}, \dots, \PT{y}_{M,K} ) \DEF \prod_{m=1}^M \prod_{k=1}^K \dd \eta_{m,k}( \PT{y}_{m,k} )
\end{equation*}
define the probability product measure formed by probability measures supported on the sets $A_{1,1}, \dots, A_{M,K}$. Using Theorem~\ref{thm:wce.form.B} and proceeding as in the proof of Theorem~\ref{thm:rms.wce.form.B}, we get
\begin{align*}
&\mathbb{E}[ \{ \wce( \numint[Y_{M K}]; \mathcal{H}(\RK) ) \}^2 ] = \mathop{\int \cdots \int}_{A_{1,1} \times \cdots \times A_{M,K}} \{ \wce( \numint[\{ \PT{y}_{1,1}, \dots, \PT{y}_{M,K} \}]; \mathcal{H}(\RK) ) \}^2 \dd \eta( \PT{y}_{1,1}, \dots, \PT{y}_{M,K} ) \\
&\phantom{e}= \mathop{\int \cdots \int}_{A_{1,1} \times \cdots \times A_{M,K}} \Bigg[ \frac{1}{M^2 K^2} \mathop{\sum_{m=1}^M \sum_{k=1}^K} \RKmod( \PT{y}_{m,k}, \PT{y}_{m,k} ) \\
&\phantom{e=\pm}+ \frac{1}{M^2 K^2} \mathop{\sum_{m=1}^M \sum_{k=1}^K \sum_{m^\prime=1}^M \sum_{k^\prime=1}^K}_{(m,k) \neq (m^\prime,k^\prime)} \RKmod( \PT{y}_{m,k}, \PT{y}_{m^\prime,k^\prime} ) \Bigg] \dd \eta( \PT{y}_{1,1}, \dots, \PT{y}_{M,K} ) \\
&\phantom{e=}- \frac{2 W( \RK_{\mathcal{S}} )}{M K} \mathop{\sum_{m=1}^M \sum_{k=1}^K} \int_{A_{m,k}} \Bigg[ \Phi( \rho ) \,\int_{\rho}^\infty  \psidensity( r ) \dd r + \int_0^{\rho} \Phi( r ) \psidensity( r ) \dd r - W( \RK_{\mathcal{R}}, \psidensity ) \Bigg] \dd \eta_{m,k}( \rho \PT{y}^* ).
\end{align*}
The right-most double sum vanishes as can be seen by reversing the partition of unity and using \eqref{eq:mean.zero.identity}. After interchanging summation and integration, the completed quadruple sum like-wise vanishes. These observations give the simpler form
\begin{equation*}
\begin{split}
\mathbb{E}[ \{ \wce( \numint[Y_{M K}]; \mathcal{H}(\RK) ) \}^2 ]
&= \frac{1}{M^2 K^2} \mathop{\sum_{m=1}^M \sum_{k=1}^K} \int_{A_{m,k}} \RKmod( \PT{y}, \PT{y} ) \dd \eta_{m,k}( \PT{y} ) \\
&\phantom{=}- \frac{1}{M^2 K^2} \mathop{\sum_{m=1}^M \sum_{k=1}^K} \int_{A_{m,k}} \int_{A_{m,k}} \RKmod( \PT{x}, \PT{y} ) \dd \eta_{m,k}( \PT{x} ) \dd \eta_{m,k}( \PT{y} ).
\end{split}
\end{equation*}
Making use of the product forms of \eqref{eq:mathcal.K.2nd.form.all} and probability measures $\eta_{m,k}$ and \eqref{eq:reproducing.kernel.modified}, we have
\begin{equation*}
\int_{A_{m,k}} \RKmod( \PT{y}, \PT{y} ) \dd \eta_{m,k}( \PT{y} ) = \int_{\rho_{k-1}}^{\rho_k} \Phi( r ) \frac{\psidensity( r ) \dd r}{1/K} - W( \RK )
\end{equation*}
and
\begin{equation*}
\begin{split}
&\int_{A_{m,k}} \int_{A_{m,k}} \RKmod( \PT{x}, \PT{y} ) \dd \eta_{m,k}( \PT{x} ) \dd \eta_{m,k}( \PT{y} ) = \left[ \int_{\rho_{k-1}}^{\rho_k} \int_{\rho_{k-1}}^{\rho_k} \Phi( \min\{ r, \rho \} ) \frac{\psidensity( r ) \dd r}{1/K} \frac{\psidensity( \rho ) \dd \rho}{1/K} \right] \\
&\phantom{equals=\pm}\times \left[ \int_{D_{m,M}} \int_{D_{m,M}} \left( 1 - C_d \left\| \PT{x}^* - \PT{y}^* \right\| \right) \frac{\dd \sigma_d( \PT{x}^* )}{1/M} \frac{\dd \sigma_d( \PT{y}^* )}{1/M} \right] - W( \RK ).
\end{split}
\end{equation*}
We observe that the second square-bracketed expression above tends to $1$ as $M \to \infty$ because of $\| \PT{x}^* - \PT{y}^* \| \leq \diam D_{m,M} \leq c / M^{1/d}$ for $\PT{x}^*, \PT{y}^* \in D_{m,M}$ and $\sigma_d( D_{m,M} ) = 1/M$.
Hence splitting up this expression and substitution into the last formula for the expected value gives, after some straightforward rearrangement of terms, the result.
\end{proof}

\begin{rmks}
\begin{inparaenum}[\bf\itshape \upshape(A\upshape)]
\item \label{rmks.A} The small diameter constraints on $D_{1,M}, \dots, D_{M,M}$ imply that
\begin{equation} \label{eq:small.diameter.constraints.consequence}
\frac{c^\prime}{M^{1/d}} \leq \int_{D_{m,M}} \int_{D_{m,M}} \left\| \PT{x}^* - \PT{y}^* \right\| \frac{\dd \sigma_d( \PT{x}^* )}{1/M} \frac{\dd \sigma_d( \PT{y}^* )}{1/M} \leq  \diam D_{m,M} \leq \frac{c}{M^{1/d}}.
\end{equation}
Thus, one has optimal order $M^{-1-1/d}$ in the second part of the right-hand side of \eqref{eq:compartmentalized.B}.
(The lower bound follows from an argument in the proof of \cite[Theorem~25]{BrSaSlWo2013}.)
\item The right-hand side of \eqref{eq:compartmentalized.B} is at least of order $N^{-1}$ as all the square-bracketed expressions are bounded (integration with respect to probability measures); also cf. Theorems~\ref{thm:rms.wce.form.A} and \ref{thm:rms.wce.form.B}.
\item Note that
\begin{equation*}
\begin{split}
G_{k,K}
&\DEF \int_{\rho_{k-1}}^{\rho_k} \Phi( r ) \frac{\psidensity( r ) \dd r}{1/K} - \int_{\rho_{k-1}}^{\rho_k} \int_{\rho_{k-1}}^{\rho_k} \Phi( \min\{ r, \rho \} ) \frac{\psidensity( r ) \dd r}{1/K} \frac{\psidensity( \rho ) \dd \rho}{1/K} \\
&= \int_{\rho_{k-1}}^{\rho_k} \int_{\rho_{k-1}}^{\rho_k} \left[ \Phi( r ) - \Phi( \min\{ r, \rho \} ) \right] \frac{\psidensity( r ) \dd r}{1/K} \frac{\psidensity( \rho ) \dd \rho}{1/K} \geq 0,
\end{split}
\end{equation*}
since $\Phi(r)$ is a cdf. Hence both contributions to the expected value in \eqref{eq:compartmentalized.B} are non-negative.
\item Application of the first mean value theorem for integration yields
\begin{align*}
\frac{1}{K} \sum_{k=1}^K G_{k,K}
&= \frac{1}{K} \sum_{k=1}^K \left[ \Phi( \rho_k^{**} ) - \Phi( \rho_k^{*} ) \right] \leq \frac{1}{K} \sum_{k=1}^K \left[ \Phi( \rho_k ) - \Phi( \rho_{k-1} ) \right] \\
&= \frac{\Phi( \rho_K ) - \Phi( \rho_0 )}{K} = \frac{\Phi( \infty ) - \Phi( 0 )}{K} = \frac{1}{K},
\end{align*}
where $\rho_{k-1} < \rho_k^{*} < \rho_k$ and $\rho_{k-1} < \rho_k^{**} < \rho_k$ for $k = 1, \dots, K$, but also $\rho_k^{*} \leq \rho_k^{**}$ by non-negativity of $G_{k,K}$ (see previous item). The estimate leading to the telescope sum and the evaluations follow from the fact that $\Phi(r)$ is a cdf.
\item Thus, the improvement in the order of convergence gained by compartmentalization depends on how fast the convergence is in
\begin{equation*} 
\frac{1}{K} \sum_{k=1}^K \int_{\rho_{k-1}}^{\rho_k} \int_{\rho_{k-1}}^{\rho_k} \Phi( \min\{ r, \rho \} ) \frac{\psidensity( r ) \dd r}{1/K} \frac{\psidensity( \rho ) \dd \rho}{1/K} \to \int_{0}^{\infty} \Phi( r ) \psidensity( r ) \dd r \quad \text{as $K \to \infty$,}
\end{equation*}
subject to the requirement that the numbers $0 = \rho_0 < \rho_1 < \cdots < \rho_{K-1} < \rho_K = \infty$ satisfy
\begin{equation} \label{eq:rho.k.conditions}
\int_0^{\rho_k} \psidensity( r ) \dd r = \frac{k}{K}, \qquad k = 0, 1, \dots, K-1, K.
\end{equation}
For a qualitative estimate we may assume a rate of convergence of $K^{-\beta} g(K)$ ($\beta \geq 1$), where $g(K)$ may not grow faster than any power of $K$. The convergence rates of both contributions to the right-hand side of \eqref{eq:compartmentalized.B} are matched when $K$ is of order $M^{\frac{1}{\beta}\frac{1}{d}}$. This in turn would imply that an upper bound of the expected value \eqref{eq:compartmentalized.B} has order $g( N^{\frac{1}{1+\beta d}} ) \big/ N^{1 + \frac{1}{d+1/\beta}}$. The lower bound has the same order, since no cancellation can take place by non-negativity of the two contributions to the expected value \eqref{eq:compartmentalized.B}.
\end{inparaenum}
\end{rmks}

The following assertion is a consequence of these remarks.

\begin{cor} 
Under the assumptions of Theorem~\ref{thm:wce.form.B}, let $(A_{m,k}^{(M,K)})$ be a small-diameter equal mass partition of $\mathbb{R}^{d+1}$ into $N = M K$ parts of equal mass $1 / N$. If
\begin{equation*} 
\int_{0}^{\infty} \Phi( r ) \psidensity( r ) \dd r - \frac{1}{K} \sum_{k=1}^K \int_{\rho_{k-1}}^{\rho_k} \int_{\rho_{k-1}}^{\rho_k} \Phi( \min\{ r, \rho \} ) \frac{\psidensity( r ) \dd r}{1/K} \frac{\psidensity( \rho ) \dd \rho}{1/K} \leq C(\phi,\psi) \frac{g(K)}{K^{\beta}},
\end{equation*}
where $\rho_0, \dots, \rho_K$ satisfy \eqref{eq:rho.k.conditions}, for sufficiently large $K$ for some fixed $\beta \geq 1$ and $g(K)$ a function not growing faster than any power of $K$, then $K$ is of order $M^{\frac{1}{\beta}\frac{1}{d}}$ and
\begin{equation*}
\begin{split}
&\mathbb{E}\Bigg[ \sup_{\substack{f \in \mathcal{H}(\RK), \\ \| f \|_{\RK} \leq 1}} \left| \frac{1}{M K} \sum_{m=1}^M \sum_{k=1}^K f( \PT{y}_{m,k}^{(M,K)} ) - \int_{\mathbb{R}^{d+1}} f( \PT{y} ) \psi( \PT{y} ) \dd \lambda_{d+1}( \PT{y} ) \right|^2 \Bigg] \leq C^\prime(\phi,\psi) \frac{g( N^{\frac{1}{1+\beta d}} )}{N^{1 + \frac{1}{d+1/\beta}}},
\end{split}
\end{equation*}
where $\PT{y}_{m,k}^{(M,K)}$ is chosen randomly from $A_{m,k}^{(M,K)}$ with respect to the probability measure $\eta_{m,k}^{(M,K)}$ induced by the density function $\psi$ (i.e., $\dd \eta_{m,k}^{(M,K)}( \rho \PT{y}^* ) = K M \dd \sigma_d|_{D_{m,M}}( \PT{y}^* ) \psidensity( \rho ) \dd \rho |_{[\rho_{k-1}, \rho_k)}$). The expected value satisfies an analogue lower bound when the first inequality can be reversed. 
\end{cor}

We observe that the bound for the expected value above (that is, the worst case error of a typical $N$-point sample chosen according to the compartmentalization strategy) approaches the lower bound $N^{-1-1/d}$ (cf. Theorem~\ref{thm:wce.lower.bound}) as $\beta$ becomes large.

Theorem~\ref{thm:compartmentalized.B} for $K = 1$ provides the worst-case error behavior for a typical sample of $M$ points $\rho_1 \PT{y}_1^*, \dots, \rho_M \PT{y}_M^* \in \mathbb{R}^{d+1}$ such that in each section
\begin{equation*}
A_m^{(M)} \DEF A_{m,1}^{(M,1)} \DEF \left\{ \rho \PT{x}^* \in \mathbb{R}^{d+1} : \PT{x}^* \in D_{m,M}, \rho \geq 0 \right\}, \qquad 1 \leq m \leq M,
\end{equation*}
exactly one point is randomly selected. The radii are independently and identically $\psidensity \lambda_1$-distributed. The points $\PT{y}_1^*, \dots, \PT{y}_M^*$ are  so-called \emph{randomized equal area points on $\mathbb{S}^d$}; i.e, the $m$th point is selected at random from $D_{m,M}$ with respect to uniform measure on $D_{m,M}$. It is shown in \cite{BrSaSlWo2013} that (cf. \eqref{eq:wce.Sobolev.sphere})
\begin{equation*}
\frac{\beta^\prime}{M^{1+1/d}} \leq \left[ \wce( \numint[\{ \PT{y}_1^*, \dots, \PT{y}_M^* \}]; \mathcal{H}( \RK_{\mathcal{S}} ) ) \right]^2 \leq \frac{\beta}{M^{1+1/d}},
\end{equation*}
where $\beta, \beta' > 0$ depend on the $\mathcal{H}( \RK_{\mathcal{S}} )$-norm, but are independent of $M$, and $\beta$ also depends on the sequence of equal area partitions $( \mathcal{D}_N)$ of $\mathbb{S}^d$ with small diameter. For the spatial variant of randomized equal area points on $\mathbb{S}^d$ we have the following corollary of Theorem~\ref{thm:compartmentalized.B}.

\begin{cor} 
Under the assumptions of Theorem~\ref{thm:wce.form.B}, let $(A_{m}^{(M)})$ be a small-diameter equal mass partition of $\mathbb{R}^{d+1}$ into $M$ parts of equal mass $1 / M$. Then
\begin{equation} \label{eq:compartmentalized.Ba}
\begin{split}
&\mathbb{E}\big[ \{ \wce( \numint[\{ \PT{y}_{1}^{(M)}, \dots, \PT{y}_{M}^{(M)} \} ]; \mathcal{H}(\RK) ) \}^2 \big] = \frac{1}{M} \left[ \int_{0}^{\infty} \Phi( r ) \psidensity( r ) \dd r - W( \RK_{\mathcal{R}}, \psidensity ) \right] \\
&\phantom{equals=}+ \frac{C_d}{M} W( \RK_{\mathcal{R}}, \psidensity ) \left[ \frac{1}{M} \sum_{m=1}^M \int_{D_{m,M}} \int_{D_{m,M}} \left\| \PT{x}^* - \PT{y}^* \right\| \frac{\dd \sigma_d( \PT{x}^* )}{1/M} \frac{\dd \sigma_d( \PT{x}^* )}{1/M} \right],
\end{split}
\end{equation}
where $\PT{y}_{m}^{(M)}$ is chosen randomly from $A_{m}^{(M)}$ with respect to the probability measure $\eta_{m}^{(M)}$ induced by the density function $\psi$ (i.e., $\dd \eta_{m}^{(M)}( \rho \PT{y}^* ) = M \dd \sigma_d|_{D_{m,M}}( \PT{y}^* ) \psidensity( \rho ) \dd \rho$).
\end{cor}

From \eqref{eq:small.diameter.constraints.consequence} (small-diameter constraints) we see that the second part of the right-hand side of the above formula is of optimal order $M^{-1 - 1/d}$.

%
%
We conclude this subsection by discussing a discrete randomized assignment process. Given a collection of pairwise different radii $\{ \rho_1, \dots, \rho_N \}$ and a configuration of $N$ pairwise different points $\{ \PT{y}_1^*, \dots, \PT{y}_N^* \}$ on $\mathbb{S}^d$, a sample of $N$ points in $\mathbb{R}^{d+1}$ can be obtained by assigning to each point $\PT{y}_j^*$ a radius $\rho_{\pi(j)}$ at random. In the ``selection without replacement'' model this $\pi$ is a permutation uniformly chosen from the symmetric group $\mathfrak{S}_N$.

\begin{thm} \label{thm:discrete.random.assignment}
Let $\{ \PT{y}_1^*, \dots, \PT{y}_N^* \}$ be collection of $N$ pairwise different points on $\mathbb{S}^d$ and $\{ \rho_1, \dots, \rho_N \}$ a family of $N$ pairwise different positive radii. Under the assumptions of Theorem~\ref{thm:wce.form.B},
\begin{equation}
\begin{split} \label{eq:discrete.random.assignment}
&\mathbb{E}\big[ \{ \wce( \numint[\{ \rho_{\pi(1)} \PT{y}_1^*, \dots, \rho_{\pi(N)} \PT{y}_N^* \} ]; \mathcal{H}(\RK) ) \}^2 \big] \\
&\phantom{equals}= \Bigg[ \frac{1}{N ( N - 1 )} \mathop{\sum_{\ell=1}^N \sum_{m=1}^N }_{\ell \neq m} \RK_{\mathcal{R}}( \rho_\ell, \rho_m ) \Bigg] \left[ \frac{1}{N^2} \sum_{i=1}^N \sum_{j=1}^N \RKmod_{\mathcal{S}}( \PT{y}_i^*, \PT{y}_j^* ) \right] \\
&\phantom{equals=}+ \frac{1}{N} \Bigg[ \frac{1}{N} \sum_{i=1}^N \RK_{\mathcal{R}}( \rho_i, \rho_i ) - \frac{1}{N ( N - 1 )} \mathop{\sum_{\ell=1}^N \sum_{m=1}^N }_{\ell \neq m} \RK_{\mathcal{R}}( \rho_\ell, \rho_m ) \Bigg] \left[ 1 - W( \RK_{\mathcal{S}} ) \right] \\
&\phantom{equals=}+ \left[ \frac{1}{N^2} \sum_{\ell=1}^N \sum_{m=1}^N \RKmod_{\mathcal{R}}( \rho_\ell, \rho_m ) - \frac{2}{N} \sum_{j=1}^N \int_0^\infty \RKmod_{\mathcal{R}}( r, \rho_j ) \psidensity( r ) \dd r \right] W( \RK_{\mathcal{S}} ),
\end{split}
\end{equation}
where $\pi$ is a permutation uniformly chosen from the symmetric group $\mathfrak{S}_N$.
\end{thm}

\begin{proof}
Let $| \mathfrak{S}_N |$ denote the cardinality $\mathfrak{S}_N$. By Theorem~\ref{thm:wce.form.B} (also cf. Proof of Theorem~\ref{thm:wce.form.B})
\begin{align*}
&\mathbb{E}\big[ \{ \wce( \numint[\{ \rho_{\pi(1)} \PT{y}_1^*, \dots, \rho_{\pi(N)} \PT{y}_N^* \} ]; \mathcal{H}(\RK) ) \}^2 \big] \\
&\phantom{equa}= \frac{1}{| \mathfrak{S}_N |} \sum_{\pi \in \mathfrak{S}_N} \{ \wce( \numint[\{ \rho_{\pi(1)} \PT{y}_1^*, \dots, \rho_{\pi(N)} \PT{y}_N^* \}]; \mathcal{H}(\RK) ) \}^2 \\
&\phantom{equa}= \frac{1}{| \mathfrak{S}_N |} \sum_{\pi \in \mathfrak{S}_N} \Bigg[ \frac{1}{N^2} \sum_{j=1}^N \mathcal{K}( \rho_{\pi(j)} \PT{y}_j^*, \rho_{\pi(j)} \PT{y}_j^*) + \frac{1}{N^2} \mathop{\sum_{i=1}^N \sum_{j=1}^N }_{i \neq j} \mathcal{K}( \rho_{\pi(i)} \PT{y}_i^*, \rho_{\pi(j)} \PT{y}_j^*) \Bigg] \\
&\phantom{equa=}- \frac{1}{| \mathfrak{S}_N |} \sum_{\pi \in \mathfrak{S}_N} \frac{2}{N} \sum_{j=1}^N \Bigg[ \int_0^\infty \RKmod_{\mathcal{R}}( r, \rho_{\pi(j)} ) \psidensity( r ) \dd r \Bigg] W( \RK_{\mathcal{S}} ).
\end{align*}
Collecting terms with the same $\pi( i )$ and $\pi( j )$, we arrive at
\begin{equation*}
\begin{split}
&\mathbb{E}\big[ \{ \wce( \numint[\{ \rho_{\pi(1)} \PT{y}_1^*, \dots, \rho_{\pi(N)} \PT{y}_N^* \} ]; \mathcal{H}(\RK) ) \}^2 \big] = \frac{1}{N} \, \frac{1}{N^2} \sum_{i=1}^N \sum_{j=1}^N \RKmod( \rho_i \PT{y}_j^*, \rho_i \PT{y}_j^* ) \\
&\phantom{e=}+ \frac{1}{N ( N - 1 )} \, \frac{1}{N^2} \mathop{\sum_{\ell=1}^N \sum_{m=1}^N }_{\ell \neq m} \mathop{\sum_{i=1}^N \sum_{j=1}^N }_{i \neq j}  \RKmod( \rho_\ell \PT{y}_i^*, \rho_m \PT{y}_j^* ) - \frac{2}{N} \sum_{j=1}^N \Bigg[ \int_0^\infty \RKmod_{\mathcal{R}}( r, \rho_j ) \psidensity( r ) \dd r \Bigg] W( \RK_{\mathcal{S}} ).
\end{split}
\end{equation*}
Substituting \eqref{eq:mathcal.K.2nd.form.all} and \eqref{eq:W.K.factorization}, we have
\begin{equation*}
\begin{split}
&\mathbb{E}\big[ \{ \wce( \numint[\{ \rho_{\pi(1)} \PT{y}_1^*, \dots, \rho_{\pi(N)} \PT{y}_N^* \} ]; \mathcal{H}(\RK) ) \}^2 \big] = \frac{1}{N} \, \frac{1}{N^2} \sum_{i=1}^N \sum_{j=1}^N \RK_{\mathcal{R}}( \rho_i, \rho_i ) \\
&\phantom{equals=}+ \Bigg[ \frac{1}{N ( N - 1 )} \mathop{\sum_{\ell=1}^N \sum_{m=1}^N }_{\ell \neq m} \RK_{\mathcal{R}}( \rho_\ell, \rho_m ) \Bigg] \Bigg[ \frac{1}{N^2} \mathop{\sum_{i=1}^N \sum_{j=1}^N }_{i \neq j} \RK_{\mathcal{S}}( \PT{y}_i^*, \PT{y}_j^* ) \Bigg] \\
&\phantom{equals=}- \frac{2}{N} \sum_{j=1}^N \Bigg[ \int_0^\infty \RK_{\mathcal{R}}( r, \rho_j ) \psidensity( r ) \dd r \Bigg] W( \RK_{\mathcal{S}} ) + W( \RK_{\mathcal{R}}, \psidensity ) W( \RK_{\mathcal{S}} ).
\end{split}
\end{equation*}
Rearranging terms, we get
\begin{align*}
&\mathbb{E}\big[ \{ \wce( \numint[\{ \rho_{\pi(1)} \PT{y}_1^*, \dots, \rho_{\pi(N)} \PT{y}_N^* \} ]; \mathcal{H}(\RK) ) \}^2 \big] = \frac{1}{N} \, \frac{1}{N} \sum_{i=1}^N \RK_{\mathcal{R}}( \rho_i, \rho_i ) \\
&\phantom{equals=}+ \Bigg[ \frac{1}{N ( N - 1 )} \mathop{\sum_{\ell=1}^N \sum_{m=1}^N }_{\ell \neq m} \RK_{\mathcal{R}}( \rho_\ell, \rho_m ) \Bigg] \left[ \frac{1}{N^2} \sum_{i=1}^N \sum_{j=1}^N \RKmod_{\mathcal{S}}( \PT{y}_i^*, \PT{y}_j^* ) - \frac{1}{N} + W( \RK_{\mathcal{S}} ) \right] \\
&\phantom{equals=}- \frac{2}{N} \sum_{j=1}^N \Bigg[ \int_0^\infty \RK_{\mathcal{R}}( r, \rho_j ) \psidensity( r ) \dd r \Bigg] W( \RK_{\mathcal{S}} ) + W( \RK_{\mathcal{R}}, \psidensity ) W( \RK_{\mathcal{S}} ) \\
&\phantom{equals}= \Bigg[ \frac{1}{N ( N - 1 )} \mathop{\sum_{\ell=1}^N \sum_{m=1}^N }_{\ell \neq m} \RK_{\mathcal{R}}( \rho_\ell, \rho_m ) \Bigg] \left[ \frac{1}{N^2} \sum_{i=1}^N \sum_{j=1}^N \RKmod_{\mathcal{S}}( \PT{y}_i^*, \PT{y}_j^* ) \right] \\
&\phantom{equals=}+ \frac{1}{N} \, \frac{1}{N} \sum_{i=1}^N \RK_{\mathcal{R}}( \rho_i, \rho_i ) - \frac{1}{N} \, \frac{1}{N ( N - 1 )} \mathop{\sum_{\ell=1}^N \sum_{m=1}^N }_{\ell \neq m} \RK_{\mathcal{R}}( \rho_\ell, \rho_m ) \\
&\phantom{equals=}+ \Bigg[ \frac{1}{N ( N - 1 )} \mathop{\sum_{\ell=1}^N \sum_{m=1}^N }_{\ell \neq m} \RKmod_{\mathcal{R}}( \rho_\ell, \rho_m ) - \frac{2}{N} \sum_{j=1}^N \int_0^\infty \RKmod_{\mathcal{R}}( r, \rho_j ) \psidensity( r ) \dd r \Bigg] W( \RK_{\mathcal{S}} ). 
\end{align*}
The result follows by rearrangement of terms.
\end{proof}

We observe that the right-hand side of \eqref{eq:discrete.random.assignment} consists of three non-negative parts:
\begin{inparaenum}[\bf\itshape \upshape(A\upshape)]
\item The first part contains the worst-case error \eqref{eq:wce.Sobolev.sphere} of a QMC method with nodes on $\mathbb{S}^d$ for $\mathbb{H}^{(d+1)/2}( \mathbb{S}^d )$, 
\item a connection term comparing the average values of the diagonal terms and the non-diagonal terms of the kernel $\RK_{\mathcal{R}}$ multiplied by $1/N$, and
\item a worst-case error as given in the following result.
\end{inparaenum}

\begin{thm}
Let $\mathcal{H}( \RK_{\mathcal{R}} )$ be the Hilbert space uniquely defined by the reproducing kernel \eqref{eq:RK.cal.R} with closed form \eqref{eq:RK.cal.R.closed.form}. Then the QMC method
\begin{equation*}
\widetilde{\numint}[\{ \rho_1, \dots, \rho_N \}]( g ) \DEF \frac{1}{N} \sum_{j=1}^N g( r_j ), \qquad g \in \mathcal{H}( \RK_{\mathcal{R}} ),
\end{equation*}
with positive radii $\rho_1, \dots, \rho_N$ approximating the exact integral
\begin{equation*}
\widetilde{\xctint}[\psidensity]( g ) \DEF \int_0^\infty g( r ) \psidensity( r ) \dd r,
\end{equation*}
where the density function $\psidensity$ is given in \eqref{eq:psi.isotropic.density.and.cdf}, has the following worst-case error representations
\begin{equation*}
\begin{split}
\wce( \widetilde{\numint}[\{ \rho_1, \dots, \rho_N \}]; \mathcal{H}( \RK_{\mathcal{R}} ) )
&= \left( \frac{1}{N^2} \sum_{\ell=1}^N \sum_{m=1}^N \RKmod_{\mathcal{R}}( \rho_\ell, \rho_m ) - \frac{2}{N} \sum_{j=1}^N \int_0^\infty \RKmod_{\mathcal{R}}( r, \rho_j ) \psidensity( r ) \dd r \right)^{1/2} \\
&= D_{\mathbb{L}_2}^{\mathcal{R}}( \{ \rho_1, \dots, \rho_N \} ).
\end{split}
\end{equation*}
Here, $D_{\mathbb{L}_2}^{\mathcal{R}}( \{ \rho_1, \dots, \rho_N \} )$ denotes the $\mathbb{L}_2$-discrepancy
\begin{equation*}
D_{\mathbb{L}_2}^{\mathcal{R}}( \{ \rho_1, \dots, \rho_N \} ) \DEF \left( \int_0^\infty \left| \widetilde{\delta}[ \{ \rho_1, \dots, \rho_N \} ]( R ) \right|^2 \phi( R ) \dd R \right)^{1/2}
\end{equation*}
of the collection $\{ \rho_1, \dots, \rho_N \}$ with local discrepancy function
\begin{equation*}
\widetilde{\delta}[ \{ \rho_1, \dots, \rho_N \} ]( R ) \DEF \frac{1}{N} \sum_{j=1}^N \indicatorfunction_{[R, \infty)}( \rho_j ) - \int_R^\infty \psidensity( \rho ) \dd \rho
\end{equation*}
with respect to half-open infinite intervals $[R, \infty)$ as test sets.
\end{thm}

\begin{proof}
The worst-case error forms can be derived similarly as in Section~\ref{subsec:wce}. We leave the details to the reader.
\end{proof}


\subsection{Normal and Nakagami distribution}
\label{sec:examples.A}

The \emph{Nakagami distribution} with shape parameter $\nu$ and spread $\Omega$ is used in engineering applications (cf., e.g., \cite{KoJiJa2004}). Its  probability density function is given by
\begin{equation} \label{eq:Nakagami.distribution}
\frac{2 \nu^\nu}{\gammafcn( \nu ) \Omega^\nu} \, x^{2\nu-1} \, \exp( - \frac{\nu}{\Omega} x^2 ), \qquad x > 0,
\end{equation}
and the corresponding cumulative distribution function is given by
\begin{equation*}
\gammafcnregularizedP( \nu, \frac{\nu}{\Omega} x^2 ),
\end{equation*}
where $\gammafcnregularizedP( a, x )$ and $\gammafcnregularizedQ( a, x )$ are the regularized incomplete gamma functions
\begin{subequations} \label{eq:regularized.gammafcn}
\begin{align}
\gammafcnregularizedP( a, x ) &\DEF \frac{\gamma(a, x)}{\gammafcn( a )}, \qquad \gamma( a, x ) \DEF \int_0^x \EulerE^{-t} t^{a-1} \dd t, \\
\gammafcnregularizedQ( a, x ) &\DEF \frac{\gammafcn(a, x)}{\gammafcn( a )}, \qquad \gammafcn( a, x ) \DEF \int_x^\infty \EulerE^{-t} t^{a-1} \dd t.
\end{align}
\end{subequations}

Suppose that the probability density function $\psi$ in \eqref{eq:psi.isotropic} is given by means of
\begin{subequations}
\begin{equation} \label{eq:h.and.Nagakami.example}
\psidensity( r ) = \omega_d h( r ) \, r^d \DEF \frac{2}{\gammafcn( \nu )} \left( \frac{\nu}{B} \right)^\nu r^{2\nu-1} \, \exp( - \frac{\nu}{B} r^2 ), \qquad r > 0,
\end{equation}
where $\nu > 0$ and $B > 0$.
%
Then
\begin{equation} \label{eq:integral.aux.A}
\int_\rho^\infty \psidensity( r ) \dd r = \frac{2}{\gammafcn( \nu )} \left( \frac{\nu}{B} \right)^\nu \int_\rho^\infty \exp( - \frac{\nu}{B} \, r^2 ) \, r^{2\nu-1} \dd r = \gammafcnregularizedQ( \nu, \frac{\nu}{B} \rho^2 ).
\end{equation}
Furthermore, we assume that for some $\mu > 0$ and $A > 0$,
\begin{equation} \label{eq:Phi.assumption}
\left( 1 - \Phi( \rho ) \right) \psidensity( \rho ) = \frac{2}{\gammafcn(\mu)} \left( \frac{\mu}{B} \right)^\mu \rho^{2\mu-1} \exp( - \frac{\mu}{A} \, \rho^2), \qquad \rho > 0,
\end{equation}
or equivalently,
\begin{equation} \label{eq:Phi.and.Nagakami.example}
1 - \Phi( \rho ) = \int_\rho^\infty \phi( r ) \dd r = \frac{\gammafcn( \nu )}{\gammafcn( \mu )} \left( \frac{\mu}{B} \right)^\mu \left( \frac{\nu}{B} \right)^{-\nu} \, \rho^{2 \mu - 2 \nu} \, \exp\big( - \big( \frac{\mu}{A} - \frac{\nu}{B} \big) \rho^2 \big), \quad \rho > 0.
\end{equation}
By definition \eqref{eq:Phi} the function $\Phi(\rho)$, $\rho \geq 0$, is a cdf with non-negative probability density function $\phi(R)$. Given that $\Phi$ satisfies \eqref{eq:Phi.and.Nagakami.example}, by assumption \eqref{eq:Phi.assumption}, we have the following additional restriction on the positive parameters $\mu, A$ and $\nu, B$; namely $\mu = \nu$ and $A < B$.

In the following let $\nu = \mu > 0$ and $0 < A < B$. Then
\begin{equation} \label{eq:Phi.and.Nagakami.example.B}
\Phi( \rho ) = \int_0^\rho \phi( r ) \dd r = 1 - \exp\big( - \mu \big( \frac{1}{A} - \frac{1}{B} \big) \rho^2 \big), \qquad \rho > 0.
\end{equation}
\end{subequations}
We need the following integral which appears in the worst-case error formula of Theorem~\ref{thm:wce.form.B}
\begin{align}
\int_0^\rho \Phi( r ) \, \psidensity( r ) \dd r
&= \int_0^\rho \psidensity( r ) \dd r - \int_0^\rho \left( 1 - \Phi( r ) \right) \psidensity( r ) \dd r  \notag \\
&= \gammafcnregularizedP\big( \mu, \frac{\mu}{B} \rho^2 \big) - \left( \frac{A}{B} \right)^\mu \gammafcnregularizedP\big( \mu, \frac{\mu}{A} \rho^2 \big). \label{eq:Phi.integral}
\end{align}
Consequently, it follows that \eqref{eq:single.integral.of.RK} can be written as
\begin{align}
&\int_{\mathbb{R}^{d+1}} \RKmod( \PT{x}, \rho \PT{y}^* ) \, \psi( \PT{x} ) \, \dd \lambda_{d+1}( \PT{x} ) \notag \\
&\phantom{equals}= W( \RK_{\mathcal{S}} ) \Bigg[ \Phi( \rho ) \, \gammafcnregularizedQ( \mu, \frac{\mu}{B} \rho^2 ) + \gammafcnregularizedP\big( \mu, \frac{\mu}{B} \rho^2 \big) - \left( \frac{A}{B} \right)^\mu \gammafcnregularizedP\big( \mu, \frac{\mu}{A} \rho^2 \big) - W( \RK_{\mathcal{R}}, \psidensity ) \Bigg] \notag \\
&\phantom{equals}= W( \RK_{\mathcal{S}} ) \Bigg[ 1 - \left( 1 - \Phi( \rho ) \right) \gammafcnregularizedQ( \mu, \frac{\mu}{B} \rho^2 ) - \left( \frac{A}{B} \right)^\mu \gammafcnregularizedP\big( \mu, \frac{\mu}{A} \rho^2 \big) - W( \RK_{\mathcal{R}}, \psidensity ) \Bigg]. \label{eq:single.int.isotropic.special}
\end{align}
Furthermore, \eqref{eq:double.int.isotropic} can be evaluated as follows:
\begin{align*}
W( \RK_{\mathcal{R}}, \psidensity )
&= 1 - 2 \int_0^\infty \exp\big( - \mu \big( \frac{1}{A} - \frac{1}{B} \big) \rho^2 \big) \gammafcnregularizedQ( \mu, \frac{\mu}{B} \rho^2 ) \psidensity( \rho ) \dd \rho \\
&= 1 - \frac{4}{\gammafcn( \mu )} \left( \frac{\mu}{B} \right)^\mu \int_0^\infty \gammafcnregularizedQ( \mu, \frac{\mu}{B} \rho^2 ) \rho^{2\mu-1} \exp\big( - \frac{\mu}{A} \rho^2 ) \dd \rho \\
&= 1 - \frac{2}{\gammafcn( \mu )} \int_0^\infty \gammafcnregularizedQ( \mu, x ) \, \exp( - \frac{B}{A} \, x) \, x^{\mu-1} \, \dd x.
\end{align*}
We then use \cite[Eq.~8.14.6]{NIST:DLMF} to express the integral in terms of a Gaussian hypergeometric function. We have
\begin{align*}
\frac{1}{\gammafcn(\mu)} \int_0^\infty \gammafcnregularizedQ( \mu, x ) \, \exp( - \frac{B}{A} \, x) \, x^{\mu-1} \, \dd x
&= \frac{1}{\gammafcn(\mu)} \frac{\gammafcn(2\mu)}{\mu \, \gammafcn(\mu)} \left( 1 + \frac{B}{A} \right)^{-2\mu} \Hypergeom{2}{1}{1,2\mu}{1+\mu}{\frac{B/A}{1+B/A}}.
\end{align*}
On observing that the \emph{regularized incomplete beta function}, defined by (cf. \cite[Eq~8.17.2]{NIST:DLMF})
\begin{equation} \label{eq:regularized.incomplete.beta.function}
\incompletebetafcnregularized_x( a, b ) \DEF \frac{\betafcn_x( a, b )}{\betafcn( a, b )}, \qquad 0 \leq x \leq 1 \, (a,b > 0)
\end{equation}
where (cf. \cite[Eq.s~8.17.1 and 8.17.3]{NIST:DLMF})
\begin{equation}
\betafcn_x( a, b ) \DEF \int_0^x t^{a-1} \left( 1 - t \right)^{b-1} \dd t, \qquad 0 \leq x \leq 1 \, (a,b > 0)
\end{equation}
and
\begin{equation}
\betafcn( a, b ) = \frac{\gammafcn(a) \gammafcn(b)}{\gammafcn(a + b)}, \qquad a,b > 0,
\end{equation}
has the hypergeometric function representation (cf. \cite[Eq.~8.17.8]{NIST:DLMF})
\begin{equation}
\incompletebetafcnregularized_x( a, b ) = \frac{\gammafcn(a + b)}{a \gammafcn(a) \gammafcn(b)} \, x^a  \left( 1 - x \right)^b \Hypergeom{2}{1}{1,a+b}{a+1}{x},
\end{equation}
we arrive at
\begin{equation} \label{eq:W.RK.R.specialized}
W( \RK_{\mathcal{R}}, \psidensity ) = 1  - 2 \left( \frac{A}{B} \right)^\mu \incompletebetafcnregularized_{\frac{B/A}{1 + B/A}}( \mu, \mu ).
\end{equation}

We summarize these observations as follows.

\begin{thm} \label{thm:wce.form.C}
Let $\mathcal{H}( \RK )$ be the Hilbert space uniquely defined by the reproducing kernel \eqref{eq:mathcal.K.1st.form} with closed form \eqref{eq:mathcal.K.2nd.form.all} and the density $\psi$ be isotropic satisfying \eqref{eq:psi.isotropic} and \eqref{eq:psi.isotropic.normalization}. Suppose \eqref{eq:h.and.Nagakami.example}. Further, we assume that $\Phi(\rho)$ is given by \eqref{eq:Phi.and.Nagakami.example.B}; hence
\begin{equation*}
\RK( r \PT{x}^*, \rho \PT{y}^* ) = \left( 1 - C_d \left\| \PT{x}^* - \PT{y}^* \right\| \right) \left[ 1 - \exp\Big( - \mu \Big( \frac{1}{A} - \frac{1}{B} \Big) \min\{ r^2, \rho^2 \} \Big) \right],
\end{equation*}
$r, \rho \geq 0$, $\PT{x}^*, \PT{y}^* \in \mathbb{S}^d$, where the parameters $\mu$, $A$ and $B$ satisfy $\mu > 0$ and $0 < A < B$. For a method $\numint[X_N]$ with node set $X_N = \{ \PT{x}_1, \dots, \PT{x}_N \} \subseteq \mathbb{R}^{d+1}$ one has
\begin{align*}
&\wce( \numint[X_N]; \mathcal{H}(\RK) ) = \Bigg( \frac{1}{N^2} \mathop{\sum_{i=1}^N \sum_{j=1}^N} \RKmod( \PT{x}_i, \PT{x}_j ) \\
&\phantom{eq}- \frac{2}{N} \sum_{j=1}^N \Bigg[ 1 - \left( 1 - \Phi( \| \PT{x}_j \| ) \right) \gammafcnregularizedQ( \mu, \frac{\mu}{B} \| \PT{x}_j \|^2 ) - \left( \frac{A}{B} \right)^\mu \gammafcnregularizedP\big( \mu, \frac{\mu}{A} \| \PT{x}_j \|^2 \big) - W( \RK_{\mathcal{R}}, \psidensity ) \Bigg] W( \RK_{\mathcal{S}} ) \Bigg)^{1/2}.
\end{align*}
The functions $\gammafcnregularizedP(a,x)$, $\gammafcnregularizedQ(a,x)$ are the regularized incomplete gamma functions given in~\eqref{eq:regularized.gammafcn}.
\end{thm}

The root mean square error of the QMC method for typical node sets reads now as follows.

\begin{thm} \label{thm:rms.wce.form.C}
Under the assumptions of Theorem~\ref{thm:wce.form.C},
\begin{equation*}
\begin{split}
&\sqrt{\mathbb{E}\big[ \{ \wce( \numint[\{ \PT{y}_1, \dots, \PT{y}_N \}]; \mathcal{H}(\RK) ) \}^2 \big]} \\
&\phantom{equals}= \frac{1}{\sqrt{N}} \left( 1 - \left( \frac{A}{B} \right)^\mu - \left[ 1 - C_d \, W( \mathbb{S}^d ) \right] \left[ 1 - 2 \left( \frac{A}{B} \right)^\mu \incompletebetafcnregularized_{\frac{1}{1 + A/B}}( \mu, \mu ) \right] \right)^{1/2},
\end{split}
\end{equation*}
where the points $\PT{y}_1, \dots, \PT{y}_N$ are independently and identically $\psi \lambda_{d+1}$-distributed in $\mathbb{R}^{d+1}$.
\end{thm}

\begin{proof}
Under the assumptions of Theorem~\ref{thm:wce.form.C}, by \eqref{eq:Phi.integral},
\begin{equation} \label{eq:aux.A}
\int_{\mathbb{R}^{d+1}} \Phi( \| \PT{x} \| ) \, \psi( \PT{x} ) \dd \lambda_{d+1}( \PT{x} ) = \int_0^\infty \Phi( r ) \psidensity( r ) \dd r = 1 - \left( \frac{A}{B} \right)^\mu.
\end{equation}
Hence, by Theorem~\ref{thm:rms.wce.form.A}, \eqref{eq:W.K.factorization} and \eqref{eq:W.RK.R.specialized},
\begin{equation*}
\begin{split}
&\int_{\mathbb{R}^{d+1}} \Phi( \| \PT{x} \| ) \, \psi( \PT{x} ) \dd \lambda_{d+1}( \PT{x} ) - W( \RK ) \\
&\phantom{equals}= 1 - \left( \frac{A}{B} \right)^\mu - \left[ 1 - C_d \, W( \mathbb{S}^d ) \right] \left[ 1 - 2 \left( \frac{A}{B} \right)^\mu \incompletebetafcnregularized_{\frac{1}{1 + A/B}}( \mu, \mu ) \right].
\end{split}
\end{equation*}
The right-hand side above is positive for $\mu/A > \nu/B$ and $\mu \geq \nu$. This can be seen from the following observations: By \eqref{eq:C.d} and \eqref{eq:W.S.d} the sequence $(a_d)$ with $a_d = C_d W(\mathbb{S}^d)$ is strictly decreasing and $a_2 = 1 / 3$. Furthermore, the regularized incomplete beta function $\incompletebetafcnregularized_x(a,a)$ is strictly increasing in $x$ on $(0,1)$ and $\incompletebetafcnregularized_{\frac{1}{1 + A/B}}( \mu, \mu ) \geq 1/2$.
\end{proof}

The analogue of Theorem~\ref{thm:rms.wce.form.B} is the following.
\begin{cor} \label{cor:rms.wce.form.C}
Under the assumptions of Theorem~\ref{thm:wce.form.C},
\begin{equation*}
\begin{split}
&\mathbb{E}\big[ \{ \wce( \numint[\{ \rho_1 \PT{y}_1^*, \dots, \rho_N \PT{y}_N^* \} ]; \mathcal{H}(\RK) ) \}^2 \big] \\
&\phantom{equals}= \frac{1}{N} \left( \frac{A}{B} \right)^\mu \left[ 2 \incompletebetafcnregularized_{\frac{1}{1 + A/B}}( \mu, \mu ) - 1 \right] + W( \RK_{\mathcal{R}}, \psidensity ) \left[ \frac{1}{N^2} \mathop{\sum_{i=1}^N \sum_{j=1}^N} \RKmod_{\mathcal{S}}( \PT{y}_i^*, \PT{y}_j^* ) \right],
\end{split}
\end{equation*}
where $\PT{y}_1^*, \dots, \PT{y}_N^* \in \mathbb{S}^{d}$ are fixed and the radii $\rho_1, \dots, \rho_N$ are independently and identically $\psidensity( r ) \lambda_{1}$-distributed. $W( \RK_{\mathcal{R}}, \psidensity )$ is given in \eqref{eq:W.RK.R.specialized}.

\end{cor}

\begin{proof}
The result is a consequence of Theorem~\ref{thm:rms.wce.form.B} and \eqref{eq:W.RK.R.specialized} and \eqref{eq:aux.A}.
\end{proof}

When compartmentalizing the selection of random points, we get the following analogue of Theorem~\ref{thm:compartmentalized.B}. Here, we only provide an asymptotic relation giving the order of the dominant term. 
For the statement of the result we make use of the notation $a_n \asymp b_n$, which means that there are numbers $c_1$ and $c_2$ such that $c_1 a_n \leq b_n \leq c_2 a_n$ for sufficiently large $n$.

\begin{cor} \label{cor:compartmentalized.C}
Under the assumptions of Theorem~\ref{thm:wce.form.C}, let $(A_{m,k}^{(M,K)})$, where $K \asymp M^{1/d}$ as $M \to \infty$, be a small-diameter equal mass partition of $\mathbb{R}^{d+1}$ into $M K$ parts of equal mass. Then 
\begin{equation*} 
\mathbb{E}\Bigg[ \sup_{\substack{f \in \mathcal{H}(\RK), \\ \| f \|_{\RK} \leq 1}} \left| \frac{1}{M K} \sum_{m=1}^M \sum_{k=1}^K f( \PT{y}_{m,k}^{(M,K)} ) - \int_{\mathbb{R}^{d+1}} f( \PT{y} ) \psi( \PT{y} ) \dd \lambda_{d+1}( \PT{y} ) \right|^2 \Bigg] \asymp \frac{1}{(M K)^{1+1/(d+1)}}, 
\end{equation*}
where $\PT{y}_{m,k}^{(M,K)}$ is chosen randomly from $A_{m,k}^{(M,K)}$ with respect to the probability measure $\eta_{m,k}^{(M,K)}$ induced by the density function $\psi$; that is,
\begin{equation*} 
\dd \eta_{m,k}^{(M,K)}( \rho \PT{y}^* ) = K M \dd \sigma_d|_{D_{m,M}}( \PT{y}^* ) \frac{2}{\gammafcn( \mu )} \left( \frac{\mu}{B} \right)^\mu \rho^{2\mu-1} \, \exp( - \frac{\mu}{B} \rho^2 ) \dd \rho |_{[\rho_{k-1}, \rho_k)}.
\end{equation*}
\end{cor}

\begin{proof}
This result follows from Theorems~\ref{thm:compartmentalized.B} using the explicit kernel given in Theorem~\ref{thm:wce.form.C}. Application of Euler-MacLaurin summation enables us to derive the leading order term of the asymptotics for large $N$. 

First observe, that the radii $0 = \rho_0 < \rho_1 < \cdots < \rho_{K-1} < \rho_K = \infty$ are defined by 
\begin{equation*}
\frac{1}{K} = \int_{\rho_{k-1}}^{\rho_k} \psidensity( r ) \dd r = \gammafcnregularizedQ( \mu, \frac{\mu}{B} \rho_{k-1}^2 ) - \gammafcnregularizedQ( \mu, \frac{\mu}{B} \rho_{k}^2 ), \qquad k = 1, \dots, K.
\end{equation*}
That is, we can write
\begin{equation} \label{eq:aux.B}
\gammafcnregularizedQ( \mu, \frac{\mu}{B} \rho_{k}^2 ) = 1 - \frac{k}{K}, \qquad \frac{\mu}{B} \rho_{k}^2 = \inversegammafcnregularizedQ( \mu, 1 - \frac{k}{K} ), \qquad\quad k = 0, \dots, K,
\end{equation}
where $\inversegammafcnregularizedQ(\mu,s)$ is the \emph{inverse regularized incomplete gamma function}, which gives the solution for $z$ in $s = \gammafcnregularizedQ(\mu,z)$ (cf. \eqref{eq:regularized.gammafcn}).
Next, direct and straightforward computation shows that (cf. \eqref{eq:double.int.isotropic} and \eqref{eq:integral.aux.A})
\begin{align*}
F_{k,K}
&\DEF \int_{\rho_{k-1}}^{\rho_k} \int_{\rho_{k-1}}^{\rho_k} \left( 1 - \Phi( \min\{ r, \rho \} ) \right) \frac{\psidensity( r ) \dd r}{1/K} \frac{\psidensity( \rho ) \dd \rho}{1/K} \\
&= K^2 \, 2 \int_{\rho_{k-1}}^{\rho_k} \left( 1 - \Phi( \rho ) \right) \left\{ \int_\rho^{\rho_k} \psidensity( r ) \dd r \right\} \psidensity( \rho ) \dd\rho \\
&= K^2 \, 2 \int_{\rho_{k-1}}^{\rho_k} \left( 1 - \Phi( \rho ) \right) \left\{ \int_\rho^\infty \psidensity( r ) \dd r - \gammafcnregularizedQ( \mu, \frac{\mu}{B} \rho_{k}^2 ) \right\} \psidensity( \rho ) \dd\rho.
\end{align*}
Hence, by \eqref{eq:aux.B} and using \eqref{eq:Phi.integral},
\begin{align*}
\frac{1}{K} \sum_{k=1}^{K} F_{k,K}
&= K \Bigg[ 2 \int_{0}^{\infty} \left( 1 - \Phi( \rho ) \right) \left\{ \int_\rho^\infty \psidensity( r ) \dd r \right\} \psidensity( \rho ) \dd\rho \\
&\phantom{=}- 2 \left( \frac{A}{B} \right)^\mu \sum_{k=1}^{K} \left( 1 - \frac{k}{K} \right) \left\{ \gammafcnregularizedP\big( \mu, \frac{\mu}{A} \rho_{k}^2 \big)  - \gammafcnregularizedP\big( \mu, \frac{\mu}{A} \rho_{k-1}^2 \big) \right\} \Bigg].
\end{align*}
By \eqref{eq:double.int.isotropic.B} and \eqref{eq:W.RK.R.specialized} and rearrangement of terms
\begin{align}
\frac{1}{K} \sum_{k=1}^{K} F_{k,K}
&= K \Bigg[ 2 \left( \frac{A}{B} \right)^\mu \incompletebetafcnregularized_{\frac{B}{A + B}}( \mu, \mu ) - 2 \left( \frac{A}{B} \right)^\mu \frac{1}{K} \sum_{k=0}^{K-1} \gammafcnregularizedP\big( \mu, \frac{\mu}{A} \rho_{k}^2 \big) \Bigg] \notag \\
&= K \Bigg[ 2 \left( \frac{A}{B} \right)^\mu \frac{1}{K} \sum_{k=0}^{K-1} \gammafcnregularizedQ\big( \mu, \frac{B}{A} \inversegammafcnregularizedQ\big( \mu, 1 - \frac{k}{K} \big) \big) - 2 \left( \frac{A}{B} \right)^\mu \incompletebetafcnregularized_{\frac{A}{A + B}}( \mu, \mu ) \Bigg].
\end{align}
Furthermore, from \eqref{eq:Phi.integral},
\begin{equation*}
\frac{1}{K} \sum_{k=1}^K \int_{\rho_{k-1}}^{\rho_k} \left( 1 - \Phi( r ) \right) \frac{\psidensity( r ) \dd r}{1/K} = \int_0^\infty \left( 1 - \Phi( r ) \right) \psidensity( r ) \dd r = \left( \frac{A}{B} \right)^\mu.
\end{equation*}
These observations lead to
\begin{align*}
\Delta_K
&\DEF \frac{1}{K} \sum_{k=1}^K \left[ \int_{\rho_{k-1}}^{\rho_k} \Phi( r ) \frac{\psidensity( r ) \dd r}{1/K} - \int_{\rho_{k-1}}^{\rho_k} \int_{\rho_{k-1}}^{\rho_k} \Phi( \min\{ r, \rho \} ) \frac{\psidensity( r ) \dd r}{1/K} \frac{\psidensity( \rho ) \dd \rho}{1/K} \right] \\
&= \frac{1}{K} \sum_{k=1}^{K} F_{k,K} - \frac{1}{K} \sum_{k=1}^K \int_{\rho_{k-1}}^{\rho_k} \left( 1 - \Phi( r ) \right) \frac{\psidensity( r ) \dd r}{1/K} \\
&= 2 \left( \frac{A}{B} \right)^\mu \Bigg[ \sum_{k=0}^{K-1} \gammafcnregularizedQ\big( \mu, \frac{B}{A} \inversegammafcnregularizedQ\big( \mu, 1 - \frac{k}{K} \big) \big) - K \, \incompletebetafcnregularized_{\frac{A}{A + B}}( \mu, \mu ) - \frac{1}{2} \Bigg].
\end{align*}
%


Application of the Euler-MacLaurin summation formula (see Appendix~\ref{appdx:C}) yields the up to second order exact asymptotics
\begin{equation*}
\frac{1}{K} \sum_{k=0}^{K-1} \gammafcnregularizedQ\big( \mu, \frac{B}{A} \inversegammafcnregularizedQ\big( \mu, 1 - \frac{k}{K} \big) \big) \sim \incompletebetafcnregularized_{\frac{A}{A + B}}( \mu, \mu ) + \frac{1}{2} \frac{1}{K} + \frac{\frac{1}{2} \left( \frac{B}{A} \right)^\mu c_d(A,B,\mu)}{K^2} \qquad \text{as $K\to \infty$}
\end{equation*}
with $c_d(A,B,\mu) = 1/6$ and therefore the asymptotic formulas
\begin{equation*}
\Delta_K \sim \frac{c_d(A,B,\mu)}{K} \qquad \text{as $K\to \infty$}
\end{equation*}
and
\begin{equation*}
\frac{1}{K} \sum_{k=1}^{K} F_{k,K} \sim \left( \frac{A}{B} \right)^\mu + \frac{c_d(A,B,\mu)}{K} \qquad \text{as $K\to \infty$}
\end{equation*}
and
\begin{equation*}
\begin{split}
&\frac{1}{K} \sum_{k=1}^K \int_{\rho_{k-1}}^{\rho_k} \int_{\rho_{k-1}}^{\rho_k} \Phi( \min\{ r, \rho \} ) \frac{\psidensity( r ) \dd r}{1/K} \frac{\psidensity( \rho ) \dd \rho}{1/K} \\
&\phantom{equals}= 1 - \frac{1}{K} \sum_{k=1}^{K} F_{k,K} \sim 1 - \left( \frac{A}{B} \right)^\mu - \frac{c_d(A,B,\mu)}{K} \qquad \text{as $K\to \infty$.}
\end{split}
\end{equation*}
Hence we get a first order asymptotic relation for the right-hand side of \eqref{eq:compartmentalized.B} of the form (as $K \to \infty$)
\begin{equation*}
\frac{1}{M K} \, \frac{c_d(A,B,\mu)}{K} + \frac{C_d}{M K} \left[ 1 - \left( \frac{A}{B} \right)^\mu \right]  \left[ \frac{1}{M} \sum_{m=1}^M \int_{D_{m,M}} \int_{D_{m,M}} \left\| \PT{x}^* - \PT{y}^* \right\| \frac{\dd \sigma_d( \PT{x}^* )}{1/M} \frac{\dd \sigma_d( \PT{y}^* )}{1/M} \right].
\end{equation*}
Taking into account that the right-most square-bracketed expression is of optimal order $M^{-1/d}$ as $M \to \infty$ (see part~\eqref{rmks.A} of remarks after Theorem~\ref{thm:compartmentalized.B}), we arrive at 
\begin{equation*}
\frac{c_d(A,B,\mu)}{M K^2} + \frac{C_d \left[ 1 - \left( \frac{A}{B} \right)^\mu \right]}{M^{1+1/d} K}  \qquad \text{as $M, K \to \infty$}.
\end{equation*}
The relation between $M$ and $K$ can be chosen such that $M K^2 \asymp M^{1+1/d} K$ a $M, K \to \infty$, which implies that $K \asymp M^{1/d}$. Since $N = M K$, we get that $M K^2 \asymp M^{1+2/d} \asymp N^{1 + 1/(d+1)}$.
%
%
The result follows.
\end{proof}


%
%


\section{Numerical results}

In this section we present numerical results for our quadrature method. In particular, we apply it to option pricing problems. For our method the quadrature points are obtained by generating uniformly distributed points on $\mathbb{S}^{d-1}$ and then varying the distance of each point from the origin such that the resulting point set emulates a normally distributed point set in space. 
%
%
More concretely, we first generate Sobol' points in the cube $[0,1]^{d}$. The first $d-1$ components of a Sobol' point are used to generate a point on $\mathbb{S}^{d-1}$ via an area-preserving map whereas the $d$-th component provides the radial component after a transformation that uses the $\chi$ distribution. 
%
%
%
In this way we obtain a uniformly distributed point set on~$\mathbb{S}^{d-1}$ and (utilizing the radial components) a normally distributed point set in space.
We describe the details of this construction in the following subsection.

\subsection{Construction of points for our method} \label{subsec_construction}

\subsubsection*{Construction of points on $\mathbb{S}^{d-1}$}

In the following we describe the mapping from the unit cube $[0,1)^{d-1}$ to $\mathbb{S}^{d-1}$. We need the {\em regularized incomplete beta function} given by
\begin{equation} \label{eq:def.regularized.incomplete.beta.function}
\betaIregularized_z( a, b ) = \betaB_z( a, b ) / \betaB( a, b ), \quad \betaB( a, b ) = \betaB_1( a, b ), \quad \betaB_z( a, b ) = \int_0^z u^{a-1} \left( 1 - u \right)^{b-1} \dd u.
\end{equation}
For integers $d \ge 3$ we define the function $h_d:[0,1] \to [0,1]$ by means of $h_d(x) \DEF I_x(d/2, d/2)$ and denote its inverse function by $h_d^{-1}$. 
%
Then we define the mapping $T:[0,1)^{d-1} \to \mathbb{S}^{d-1}$, $T(\PT{x}) = \PT{y}_{(d-1)}$, inductively as follows: 
given $\PT{x} = (x_1, x_2, \ldots, x_{d-1}) \in [0,1)^{d-1}$, 
\begin{equation}
\begin{aligned} \label{eq:mapping.T}
x_1 \mapsto \PT{y}_{(1)} 
&=  (\cos (2\pi x_1), \sin (2\pi x_1) ), \\
(x_1, x_2) \mapsto \PT{y}_{(2)} 
&= \left(\sqrt{1 - (1-2x_2)^2} \, \PT{y}_{(1)}, 1-2 x_2 \right), \\
(x_1, x_2, x_3) \mapsto \PT{y}_{(3)} 
&= \left(\sqrt{1-(1-2 h_3^{-1}(x_3))^2 } \, \PT{y}_{(2)}, 1-2 h_3^{-1}(x_3) \right), \\
&\phantom{e}\vdots \\
(x_1, x_2, \ldots, x_{d-1}) \mapsto \PT{y}_{(d-1)} 
&= \left(\sqrt{1-(1- 2 h_{d-1}^{-1}(x_{d-1}))^2 } \, \PT{y}_{(d-2)}, 1- 2 h_{d-1}^{-1}(x_{d-1}) \right).
\end{aligned}
\end{equation}
In Appendix~\ref{appdx:B} we show that the transformation $T$ is area preserving. In particular, if $\PT{x}$ is uniformly distributed in $[0,1)^d$, then $T(\PT{x})$ is uniformly distributed on $\mathbb{S}^{d-1}$.

\subsubsection*{Points in $\mathbb{R}^{d}$}

To obtain points in $\mathbb{R}^d$ which have standard normal distribution, we use the mapping $\Phi: [0,1)^d \to \mathbb{R}^d$ given by
\begin{equation*}
\Phi(\PT{x}) \DEF F_d^{-1}(x_d) \, T(x_1,\ldots, x_{d-1}), \qquad \PT{x} = ( x_1, \dots, x_d ) \in [0,1)^d,
\end{equation*}
where $F_d^{-1} :[0,1) \to [0, \infty)$ is the inverse cdf of the $\chi$-distribution with $d$ degrees of freedom. This $\chi$-distribution is a special case of the Nakagami distribution with shape parameter $\nu = d/2$ and spread $\Omega = d$, see \eqref{eq:Nakagami.distribution}, with probability density function
\begin{equation*}
f_d(x) = \frac{2^{1-d/2}}{\Gamma(d/2)} \, x^{d-1} \EulerE^{x^2/2}
\end{equation*}
and cumulative distribution function 
\begin{equation*}
F_d(x) = \gammafcnregularizedP( d/2, x^2 / 2 )
\end{equation*}
expressed in terms of the regularized incomplete gamma functions $\gammafcnregularizedP(a,b)$ given in \eqref{eq:regularized.gammafcn}.

\subsubsection*{Sobol' points}

In order to obtain a point set in $\mathbb{R}^d$ with normal distribution, we first generate Sobol' points $\PT{x}_1, \PT{x}_2, \ldots, \PT{x}_N \in [0,1)^d$ and then set $\PT{y}_n = \Phi(\PT{x}_n)$ for $1 \leq n \le N$. 

\subsection{A trial integral on the sphere}

In the following we compare the performance of our quadrature point construction (see \eqref{eq:mapping.T}) with two standard constructions, namely:
\begin{description}
\item[Inverse normal cdf] It is well-known that the normalized random vector 
\begin{equation*}
(Z_1, \dots, Z_d) / \sqrt{Z_1^2 + \cdots + Z_d^2}
\end{equation*}
is uniformly distributed on $\mathbb{S}^{d-1}$ for a collection of $d$ random variables $Z_1, \dots, Z_d$ that are independent and identically standard normal distributed. 
Utilizing this fact, a standard method to construct quadrature points on $\mathbb{S}^{d-1}$ is by mapping a well-distributed set in $(0,1)^d$ to $\mathbb{R}^d$ using the inverse standard normal cdf for each point coordinate and subsequently normalize each point so that it lies on $\mathbb{S}^{d-1}$. For our numerical result we use  scrambled Sobol' point sets \cite{O98} in $(0,1)^d$.
\item[Random points on $\mathbb{S}^{d-1}$] A collection of random points in $[0,1)^d$ is mapped to $\mathbb{S}^{d-1}$ using the transformation $T$ from Section~\ref{subsec_construction}. We use matlab functionality to generate pseudo random point sets in $[0,1)^d$ and then map them to the sphere using~$T$.
\end{description}
We numerically approximate the exact integral 
\begin{equation*}
I \DEF \int_{\mathbb{S}^d} f(\PT{x}) \dd \sigma_d(\PT{x})
\end{equation*}
using the equal weight quadrature rule
\begin{equation*}
\frac{1}{N} \sum_{n=1}^N f(\PT{x}_n).
\end{equation*}
%
As trial function we choose $f(\PT{x}) = (x_1 + x_2 + \cdots + x_d)^2$. Then we have
\begin{equation*}
I = \int_{\mathbb{S}^d} (x_1 + x_2 + \cdots + x_d)^2 \dd \sigma_d(\PT{x}) = \int_{\mathbb{S}^d} \left(1+ 2 \sum_{1 \le i < j \le d} x_i x_j \right) \dd \sigma_d(\PT{x}) = 1,
\end{equation*}
where the last step follows by symmetry and the fact that $\int_{\mathbb{S}^d} \dd \sigma_d = 1$.

In Tables~\ref{table: sphere16}, \ref{table: sphere32} and \ref{table: sphere64} we present the integration error
\begin{equation*}
\left| I - \frac{1}{N} \sum_{n=1}^N f(\PT{x}_n) \right|
\end{equation*}
for each of the three constructions: the first uses the inverse beta function and Sobol' points from Section~\ref{subsec_construction}, the second uses a normalization of Sobol' points transformed to $\mathbb{R}^d$ via the inverse normal cdf, and the third uses random points on the sphere.

{
\footnotesize
\begin{table}[htb]
\caption{\bf$d = 16$}\label{table: sphere16}
\begin{tabular}{rccc}
\hline
           & Inverse beta function & Inverse normal cdf & Random points \\
\hline
      1024 &    1.95E-02 &   4.32E-02 &  4.65E-02 \\

      4096 &    5.67E-03 &   1.36E-03 &  1.36E-02 \\

     16384 &   3.82E-03 &   4.01E-03 & 1.51E-02 \\

     65536 &   8.89E-04 &   8.54E-04 & 2.67E-03 \\

    262144 &  1.25E-04 &   2.22E-04 & 1.76E-03 \\

   1048576 & 6.08E-05 &  4.35E-05 &  5.71E-04 \\
\hline
\end{tabular}
\end{table}

\begin{table}[htb]
\caption{\bf$d = 32$}\label{table: sphere32}
\begin{tabular}{rccc}
\hline
           & Inverse beta function & Inverse normal cdf & Random points \\
\hline
      1024 & 5.75E-02 & 7.05E-02 & 4.08E-02 \\

      4096 & 1.22E-02 & 2.99E-02 & 2.68E-02 \\

     16384 &2.35E-04 & 5.27E-03 & 2.76E-02 \\

     65536 &1.04E-03 & 1.07E-03 & 9.51E-03 \\

    262144 &2.47E-04 &7.85E-04 & 1.81E-03 \\

\hline
\end{tabular}
\end{table}

\begin{table}[htb]
\caption{\bf$d = 64$}\label{table: sphere64}
\begin{tabular}{rccc}
\hline
           & Inverse beta function & Inverse normal cdf & Random points \\
\hline
      1024 & 1.86E-02 & 6.84E-03 & 3.11E-02 \\

      4096 & 2.23E-02 & 5.86E-03 & 3.11E-03 \\

     16384 & 1.05E-02 &1.58E-02 & 8.61E-03 \\

     65536 & 2.16E-03 &4.83E-03 & 3.56E-04 \\


\hline
\end{tabular}
\end{table}
\normalsize
}

In these numerical approximations the first construction usually yields the best result followed by the inverse normal cdf construction and the random points.


\subsection{Option pricing problems}

We use our numerical scheme now to approximate option prices and compare it to standard Monte Carlo and Quasi-Monte Carlo approximations (see for instance \cite{Gla04, HeWa14, Leo12, Wa12, WaSl11} for more background on the numerics of option pricing).

We consider the problem of pricing several types of options, where the underlying asset price $S_t$ is driven by a geometric Brownian motion with SDE
\begin{equation*}
\dd S_t = \mu S_t \dd t + \sigma S_t \dd B_t,
\end{equation*}
where $\mu$ is the mean growth rate, $\sigma$ the volatility and $B_t$ is a standard Brownian motion. For simplicity we assume that the asset prices are observed at equally spaces times $t_j = j \Delta t$ for $j = 1, 2, \ldots, d$ with $\Delta t = T/d$ and where $T$ is the time at the expiration date.

\subsubsection*{Arithmetic Asian Call Option}

%

The payoff of a discrete arithmetic Asian call option is
\begin{equation}\label{payoff}
C_A = \max\big\{ \overline{S} - K, 0 \big\},
\end{equation}
where $K$ is the strike price and $$\overline{S} = \frac{1}{d}\sum_{j=1}^{d}S_{t_j}$$ is the arithmetic mean of equally time-spaced underlying asset prices at times $t_j = j \Delta t$. According to the principle of risk-neutral valuation,
the price of such an option could be presented as (see \cite{Hull12})
\begin{equation*}
C_A = \mathbb{E}_\mathbb{Q} \big[ \EulerE^{-\mu T}\max\big\{ \overline{S} - K, 0 \big\} \big],
\end{equation*}
where $\mathbb{E}_\mathbb{Q}[\cdot]$ is the expectation under the risk-neutral measure $\mathbb{Q}$.
Under $\mathbb{Q}$, the asset price at time $t_j$ is
\begin{equation}\label{eq_Stj}
S_{t_j} = S_0 \, \exp\Big( \Big(\mu -\frac{\sigma^2}{2} \Big) t_j + \sigma B_{t_j} \Big),
\end{equation}
where $\mu$ is the risk-free rate, $S_0$ is the asset price at time $t=0$ and $\left(B_{t_1}, \ldots, B_{t_d}\right)^T \sim N(\bold{0}, \Sigma)$, where the components of the covariance matrix are given by $\Sigma(i,j) = \Delta t \min(t_i, t_j)$.

One way to generate the set of random variables $B_{t_j}$ is by a random walk construction, where $B_{t_0} = 0$ and
\begin{align}\label{random_walk}
B_{t_j} =& B_{t_{j-1}} + \sqrt{\Delta t} \, Z_j, \quad \mbox{for } j = 1,\ldots, d,
\end{align}
where $Z_1, \ldots, Z_d$ are independent standard normal random variables. Using \eqref{eq_Stj} we can randomly generate asset prices $S_{t_j}$ for $j = 1, 2, \ldots, d$ and thus obtain an estimation of the payoff \eqref{payoff}. 

There are also variations of the standard random walk construction \eqref{random_walk}, which often perform better in combination with deterministic sampling methods. To obtain other possible constructions, we can view the vector $(B_{t_1}, B_{t_2}, \ldots, B_{t_d})$ as a vector whose components are normally distributed, each component having mean $0$ and the vector of random variables has covariance matrix
\begin{equation*}
\Sigma =  \Delta t
\begin{pmatrix}
1 & 1 & 1 & \ldots & 1 \\
1 & 2 & 2 & \ldots & 2 \\
1 & 2 & 3 & \ldots & 3 \\
\vdots & \vdots & \vdots & \ddots  & \vdots \\
1 & 2 & 3 & \ldots & n
\end{pmatrix}.
\end{equation*}
Then if $Z_1, \ldots, Z_d \sim N(0, 1)$ and $A$ is a matrix such that $A A^\top = \Sigma$, then the vector of random variables
\begin{equation*}
A \begin{pmatrix} Z_1 \\ \vdots \\ Z_d \end{pmatrix}
\end{equation*}
has the same mean and variance as the vector $(B_{t_1}, \ldots, B_{t_d})^\top$. The standard construction chooses
\begin{equation*}
A = \sqrt{\Delta t} \begin{pmatrix} 1 & 0 & 0 & \ldots & 0 \\ 1 & 1 & 0 & \ldots & 0 \\ \vdots & \vdots & \ddots & \ddots & \vdots \\ 1 & 1 & \ldots & 1 & 0 \\  1 & 1 & \ldots & \ldots & 1 \end{pmatrix}.
\end{equation*}

The Principle Component Analysis (PCA) construction on the other hand works the following way. Let $\lambda_1 \ge \lambda_2 \ge \cdots \ge \lambda_d \ge 0$ be the eigenvalues and $v_1, v_2, \ldots, v_d \in \mathbb{R}^d$ be the corresponding normalized eigenvectors of $\Sigma$. Then in the PCA construction, one chooses
\begin{equation*}
A = \left( \sqrt{\lambda_1} v_1, \ldots, \sqrt{\lambda_d} v_d \right).
\end{equation*}

The traditional way to generate normal random variables $Z_1, Z_2, \ldots, Z_d \sim N(0, 1)$ is by using the inverse normal cumulative distribution function $\Phi^{-1}$ and pseudo random points in $(0, 1)^d$, that is, $z_j = \Phi^{-1}(u_j)$, where $u_j \in (0,1)$. We use this method as a benchmark (termed MC (Monte Carlo) in the table below). For this method, there is no noticeable difference between the standard construction and the PCA construction and thus we only use the standard construction in this case. 
The Quasi-Monte Carlo approach replaces the pseudo-random numbers with low-discrepancy point sets $\{\PT{x}_1, \ldots, \PT{x}_N\}$, where $\PT{x}_n = (x_{1,n}, \ldots, x_{d,n}) \in (0,1)^d$. In our case we use scrambled Sobol' point sets \cite{O98} for the numerical simulations. We generate vectors $(z_{1,n}, \ldots, z_{d,n}) \in \mathbb{R}^d$ by setting $z_{j,n} = \Phi^{-1}(x_{j,n})$ and using the vectors $(z_1, \ldots, z_{d,n})$ in the standard or PCA construction. In this case we perform the numerical simulations for both, the standard construction and the PCA construction. These results serve as a second benchmark. 
The third main construction generates the points $(z_{1,n}, \ldots, z_{d,n}) \in \mathbb{R}^d$ using the method described in Section~\ref{subsec_construction}. 
It should be noted that except for the Monte Carlo method, we use scrambled Sobol' points, where we perform $128$ independent scramblings (i.e., we choose a Sobol' point set of size $N/128$ and use $128$ scramblings to generate $N$ points altogether). Scrambling has been introduced in \cite{Ow1997} and simplified versions which are easier to implement have been discussed in \cite{Hi1996,Ma1998,Ow2003}.

{
\footnotesize
\begin{table}[htb] \caption{Asian option}\label{table: Asian Option}
\begin{tabular}{r|c|c|c|c|c}
\hline
\multicolumn{ 1}{c|}{} & \multicolumn{ 3}{c}{Inverse normal} & \multicolumn{ 2}{|c}{Sphere normal} \\
\cline{2-6}
\multicolumn{ 1}{c|}{N} & \multicolumn{ 1}{c}{MC} & \multicolumn{ 1}{|c|}{Sobol' \& Standard} & \multicolumn{ 1}{c}{Sobol' \& PCA} & \multicolumn{ 1}{|c|}{Sphere \& Standard} & \multicolumn{ 1}{c}{Sphere \& PCA} \\

\hline
32768& 4.2E-02 & 1.4E-02 & 5.8E-03 & 1.5E-02 & 5.6E-03 \\
65536& 3.5E-02 & 1.0E-02 & 3.2E-03 & 1.1E-02 & 2.8E-03 \\
131072&2.4E-02 & 4.9E-03 & 1.7E-03 & 5.2E-03 & 1.8E-03 \\
262144&1.6E-02 & 2.8E-03 & 7.7E-04 & 3.2E-03 & 7.0E-04 \\
524288&1.3E-02 & 2.1E-03 & 3.8E-04 & 1.7E-03 & 3.1E-04 \\

\hline
\end{tabular}
\end{table}
\normalsize
}
In our numerical simulation, we assume that
\begin{equation*}
S_0 = 100, \quad K = 100, \quad T = 1, \quad \sigma = 0.2, \quad \mu = 0.05, \quad d = 30.
\end{equation*}
Table~\ref{table: Asian Option} shows the numerical results, which contains the standard deviation for each point set and path construction method. We observe that our sphere normal generation achieves a clear advantage over the crude Monte Carlo method with traditional normal vector generation. The results are largely similar to the Sobol' point set using the inverse cumulative distribution function (although often marginally better). We also observe that the PCA construction significantly improves both constructions using Sobol' point sets, the inverse normal cumulative distribution function method and the spherical method. This may indicate  that our sphere construction, like the construction via the inverse normal cumulative distribution function, has especially good uniform properties in the first few dimensions compared to the latter ones. This property is intrinsic in the Sobol' point set and may thus be inherited from the Sobol' point set.

\subsubsection*{Barrier Option and Digital Option}

Now we turn to more complex financial derivatives such as barrier options and digital options, both of which have discontinuous payoff functions at the terminal time $T$. Consider an up-and-out barrier Asian option, whose terminal payoff is
\begin{equation*}
C_b = \indicatorfunction_{\{S_{\max} < b\}}\, \max\{ \overline{S} - K, 0\},
\end{equation*}
where $\indicatorfunction_{\{S_{\max} < b\}}$ is the indicator function of $\{ S_{\max} < b\}$, $b$ is the knock-out barrier price and $S_{\max}$ given by $\max\{S_{t_1}, \ldots ,S_{t_d}\}$ denotes the maximum of the underlying asset during this time period.
This option behaves in every way like an Asian option, except when the underlying asset price moves above the knock-out barrier, in which case the option becomes invalid.

A digital Asian option's payoff is
\begin{equation*}
C_D = \indicatorfunction_{\{\overline{S} > K\}}.
\end{equation*}
The digital option is valid only if some condition is satisfied, and its payoff could only be $0$ or $1$, unlike the options we discussed above.

Under the Black-Scholes model, barrier options and digital options could also be priced by applying the risk-neutral valuation principle. Using the same notation as above, the prices of a barrier Asian option and a digital Asian option could be written as
\begin{subequations}
\begin{align}
C_b &= \mathbb{E}_\mathbb{Q} [\EulerE^{-rT} \, \indicatorfunction_{\{S_{\max} < b\}}\,\max\{ S_{\mathrm{ave}} - K, 0 \}],\label{eq:barrier price}\\
C_D &= \mathbb{E}_\mathbb{Q} [\EulerE^{-rT} \, \indicatorfunction_{\{S_{\mathrm{\mathrm{ave}}} > K\}}].\label{eq:digital price}
\end{align}
\end{subequations}
We simulate the asset prices in the same way as in the Asian option application.

In our numerical simulation, we assume that
\begin{equation*}
S_0 = 100, \quad K = 100, \quad b = 130, \quad T = 1, \quad \sigma = 0.2, \quad r = 0.05, \quad d = 30.
\end{equation*}
The numerical results are presented in Tables~\ref{table: Barrier Option} and \ref{table: Digital Option}.
{
\footnotesize
\begin{table}[htb] \caption{Barrier option}\label{table: Barrier Option}
\begin{tabular}{r|c|c|c|c|c}
\hline
\multicolumn{ 1}{c|}{} & \multicolumn{ 3}{c}{Inverse Normal} & \multicolumn{ 2}{|c}{Sphere Normal} \\
\cline{2-6}
\multicolumn{ 1}{c|}{N} & \multicolumn{ 1}{c}{MC} & \multicolumn{ 1}{|c|}{Sobol' \& Standard} & \multicolumn{ 1}{c}{Sobol' \& PCA} & \multicolumn{ 1}{|c|}{Sphere \& Standard} & \multicolumn{ 1}{c}{Sphere \& PCA} \\

\hline
32768&  2.0E-02 & 1.8E-02 & 1.2E-02 & 2.1E-02 & 1.1E-02 \\
65536&  1.5E-02 & 1.2E-02 & 9.0E-03 & 1.4E-02 & 6.9E-03 \\
131072& 1.0E-02 & 9.7E-03 & 6.0E-03 & 9.0E-03 & 5.3E-03 \\
262144& 7.9E-03 & 6.7E-03 & 3.4E-03 & 6.9E-03 & 3.3E-03 \\
524288& 5.1E-03 & 4.4E-03 & 2.4E-03 & 4.1E-03 & 2.2E-03 \\

\hline
\end{tabular}
\end{table}

\begin{table}[htb] \caption{Digital option}\label{table: Digital Option}
\begin{tabular}{r|c|c|c|c|c}
\hline
\multicolumn{ 1}{c|}{} & \multicolumn{ 3}{c}{Inverse Normal} & \multicolumn{ 2}{|c}{Sphere Normal} \\
\cline{2-6}
\multicolumn{ 1}{c|}{N} & \multicolumn{ 1}{c}{MC} & \multicolumn{ 1}{|c|}{Sobol' \& Standard} & \multicolumn{ 1}{c}{Sobol' \& PCA} & \multicolumn{ 1}{|c|}{Sphere \& Standard} & \multicolumn{ 1}{c}{Sphere \& PCA} \\

\hline
32768&  3.0E-03 & 1.5E-03 & 6.1E-04 & 1.5E-03 & 6.1E-04 \\
65536&  2.0E-03 & 1.0E-03 & 4.3E-04 & 1.1E-03 & 3.9E-04 \\
131072& 1.4E-03 & 7.8E-04 & 2.7E-04 & 6.7E-04 & 2.3E-04 \\
262144& 8.8E-04 & 4.7E-04 & 1.8E-04 & 5.0E-04 & 1.6E-04 \\
524288& 6.6E-04 & 3.5E-04 & 1.3E-04 & 3.4E-04 & 1.0E-04 \\

\hline
\end{tabular}
\end{table}
\normalsize
}

From the numerical results of these two applications we observe that the advantage of the constructions based on low-discrepancy point sets is diminishing and the construction based on the inverse normal distribution function performs similarly as the construction based on points on the sphere. We also observe that the improvement brought on by the PCA construction is not as large as in the previous Asian option application. This is an expected phenomenon because the barrier and digital options involve discontinuous payoffs, therefore the integrands in (\ref{eq:barrier price}) and (\ref{eq:digital price})  are discontinuous as well. In order to achieve an improved rate of convergence, the low-discrepancy methods generally require smoothness of the integrand, which is not given in these examples. 


\appendix

\section{Integral representation for $f \in \mathcal{H}(\RK)$}
\label{appdx:A}

We show that every $f \in \mathcal{H}(\RK)$ has an integral representation \eqref{eq:U.normal.form}. Let $f \in \mathcal{H}(\RK)$. Since the linear forms $\sum_{j=1}^n \alpha_j \, \RK( \PT{\cdot}, \PT{y}_j )$ lie dense in~$\mathcal{H}(\RK)$, there exists a sequence of functions
\begin{equation*}
U_n( \PT{x} ) \DEF \sum_{j=1}^n \alpha_{n,j} \, \RK( \PT{\cdot}, \PT{y}_{n,j} ) \in \mathcal{H}(\RK), \qquad n = 1, 2, 3, \dots,
\end{equation*}
such that $\| f - U_n \|_{\RK} \to 0$ as $n \to \infty$. From \eqref{eq:U.normal.form},
\begin{equation*}
U_n( \PT{x} ) = \int_0^\infty \int_{-1}^1 \int_{\mathbb{S}^d} \indicatorfunction_{\mathcal{C}( \PT{z}^*, t; R )}( \PT{x} ) u_n( \PT{z}^*, t, R )  \dd \sigma_d( \PT{z}^* ) \dd t \, \phi( R ) \dd R, \qquad \PT{x} \in \mathbb{R}^{d+1},
\end{equation*}
where the corresponding functions $u_n$ are given by
\begin{equation*}
u_n( \PT{z}^*, t, R ) \DEF \sum_{j=1}^n \alpha_{n,j} \, \indicatorfunction_{\mathcal{C}( \PT{z}^*, t; R )}( \PT{y}_{n,j} ).
\end{equation*}
Elementary algebra shows that
\begin{equation*}
\left\| U_n - U_m \right\|_{\RK}^2
= \int_0^\infty \int_{-1}^1 \int_{\mathbb{S}^d} \left| u_n( \PT{z}^*, t, R ) - u_m( \PT{z}^*, t, R ) \right|^2 \dd \sigma_d( \PT{z}^* ) \dd t \, \phi( R ) \dd R.
\end{equation*}
The sequence $(U_n)$ converges in the Hilbert space $\mathcal{H}(\RK)$ and is, thus, a Cauchy sequence satisfying $\| U_n - U_m \|_{\RK} \to 0$ as $m, n \to \infty$. Consequently, by the above relation, $(u_n)$ is a Cauchy sequence in $\mathbb{L}_2( \mathbb{S}^d \times [-1,1] \times (0,\infty); \mu_d )$ where $\dd \mu_d( \PT{z}^*, t, R ) = \dd \sigma_d( \PT{z}^* ) \dd t \, \phi( R ) \dd R$. Hence $(u_n)$ has a limit function $u \in \mathbb{L}_2( \mathbb{S}^d \times [-1,1] \times (0,\infty); \mu_d )$. Set
\begin{equation*}
U( \PT{x} ) \DEF \int_0^\infty \int_{-1}^1 \int_{\mathbb{S}^d} \indicatorfunction_{\mathcal{C}( \PT{z}^*, t; R )}( \PT{x} ) u( \PT{z}^*, t, R )  \dd \sigma_d( \PT{z}^* ) \dd t \, \phi( R ) \dd R, \qquad \PT{x} \in \mathbb{R}^{d+1}.
\end{equation*}
Then as $n \to \infty$, 
\begin{equation*}
\left\| U_n - U \right\|_{\RK}^2 = \int_0^\infty \int_{-1}^1 \int_{\mathbb{S}^d} \left| u_n( \PT{z}^*, t, R ) - u( \PT{z}^*, t, R ) \right|^2 \dd \sigma_d( \PT{z}^* ) \dd t \, \phi( R ) \dd R \to 0. 
\end{equation*}
Also, by the definition of $U_n$, $\| f - U_n \|_{\RK} \to 0$ as $n \to \infty$. Hence the right-hand side in
\begin{equation*}
0 \leq \left\| f - U \right\|_{\RK} \leq \left\| f - U_n \right\|_{\RK} + \left\| U_n - U \right\|_{\RK}, \qquad n = 1, 2, \dots,
\end{equation*}
tends to $0$ as $n \to \infty$. We conclude that $\| f - U \|_{\RK} = 0$, which shows the assertion.

\section{Mapping via cylindrical sphere coordinates}
\label{appdx:B}

We show that the mapping used in Section~\ref{subsec_construction} is area preserving. We show this result for elementary intervals $[\PT{0}, \PT{x}]$; i.e., subintervals of the form $\prod_{j=1}^d [0, x_j]$, where $\PT{x} = (x_1, \ldots, x_d) \in [0,1)^d$. The result for general Lebesgue measurable sets then follows.

By a slight abuse of notation, let $T_{(q)}( \PT{x} ) \DEF \PT{y}_{(q)}$ be the projection of $\PT{x}$ into level $q$ of~\eqref{eq:mapping.T}. We compute the surface area of $T_{(d)}([\PT{0}, \PT{x}])$. 
Set $t_q = 1 - 2 h_{q}^{-1}(x_q)$ for $2 \leq q \leq d$. 
We make use of the 'cylindrical' decomposition (cf. M{\"u}ller~\cite{Mu1966})
\begin{equation*}
\dd \sigma_1( \PT{y}_{(1)} ) = \frac{\dd \phi}{2 \pi}, \quad 
\dd \sigma_{q} ( \PT{y}_{(q)} ) = \frac{\omega_{q-1}}{\omega_{q}} \left( 1 - \tau_{q}^2 \right)^{q/2-1} \dd \tau_{q} \, \dd \sigma_{q-1}( \PT{y}_{(q-1)} ), \,\, q = 2, 3, \ldots, d 
\end{equation*}
of the normalized surface area measure $\sigma_{q}$ on $\mathbb{S}^{q}$, where $\omega_{q}$ denotes the surface area of~$\mathbb{S}^{q}$. 
This gives the recursion $\sigma_1( T_{(1)}( [0, x_1] ) ) = \int_0^{2\pi x_1} \dd \phi / (2 \pi ) = x_1$ and
\begin{equation*}
\begin{split}
\sigma_{q}( T_{(q)}( [\PT{0}, \PT{x}] ) ) 
&= \frac{\omega_{{q}-1}}{\omega_{q}} \int_{t_{q}}^1 \int_{T_{(q-1)}( [\PT{0}, \PT{x}] )} \left( 1 - \tau_{q}^2 \right)^{{q}/2-1} \dd \sigma_{{q}-1}( \PT{x}_{{q}-1} ) \dd \tau_{{q}} \\
&= \Psi_{q}( t_{q} ) \, \sigma_{{q}-1}( T_{(q-1)}( [\PT{0}, \PT{x}] ) ), \qquad q = 2, 3, \ldots, d,
\end{split}
\end{equation*}
where $\Psi_{q}( t )$ is in fact the surface area of a spherical cap $\{ \PT{z} \in \mathbb{S}^{q} : \PT{z} \cdot \PT{a} \geq t \}$, 
\begin{align*}
\Psi_{q}( t )
&= 2^{{q}-1} \frac{\omega_{{q}-1}}{\omega_{q}} \int_0^{(1-t)/2} u^{{q}/2-1} \left( 1 - u \right)^{{q}/2-1} \dd u = \betaIregularized_{(1-t)/2}({q}/2, {q}/2)
\end{align*}
in terms of the regularized incomplete beta function (see \eqref{eq:def.regularized.incomplete.beta.function}).
Note that $\Psi_2(t) = (1 - t)/2$. 
Thus 
\begin{equation*}
\sigma_d( T_{(d)}( [\PT{0}, \PT{x}] ) ) = x_1 \, x_2 \, \prod_{{q}=3}^d \betaIregularized_{h^{-1}_{q}(x_{q})}({q}/2, {q}/2) = \prod_{{q}=1}^d x_{q} = \lambda_d([\PT{0}, \PT{x}]),
\end{equation*}
which shows the result for all subintervals. (By convention, an empty product equals $1$.)

\section{Application of Euler-MacLaurin summation}
\label{appdx:C}

In this appendix we prove the following (up to second order exact) asymptotic result.

\begin{lem} \label{lem:Lambda.K}
Let $\mu > 0$ and $c > 1$. Then 
\begin{equation*}
\Lambda_K \DEF \Lambda_K( \mu, c ) \DEF \frac{1}{K} \sum_{k=1}^K \gammafcnregularizedQ\big( \mu, c \, \inversegammafcnregularizedQ\big( \mu, \frac{k}{K} \big) \big) \sim \incompletebetafcnregularized_{\frac{1}{1 + c}}( \mu, \mu ) + \frac{1}{2} \frac{1}{K} + \frac{\frac{1}{2} \, c^\mu \, c_d(\mu, c)}{K^2}
\end{equation*}
as $K \to \infty$, where $c_d(\mu,c) = 1/6$.
\end{lem}

For $\mu > 0$ and $c > 1$, we define the function
\begin{equation*}
u( x ) \DEF \gammafcnregularizedQ\big( \mu, c \, \inversegammafcnregularizedQ\big( \mu, x \big) \big), \qquad 0 \leq x \leq 1.
\end{equation*}
It can be readily seen that $\lim_{x \to 0^+} u( x ) = 0$ and $\lim_{x \to 1^-} u( x ) = 1$. The function $u$ is strictly monotonically increasing on $(0,1)$. This can be seen from the derivative 
\begin{equation*}
u^\prime( x ) = c^\mu \, \EulerE^{-(c-1) \inversegammafcnregularizedQ( \mu, x )}
\end{equation*}
which is positive on $(0,1)$ and has the value $0$ at $0$ and $c^\mu$ at $1$. We shall use the estimate
\begin{equation} \label{eq:auxiliary.estimate}
\gammafcnregularizedQ( \mu, c \, y ) = \frac{c^\mu}{\gammafcn( \mu )} \int_y^\infty \EulerE^{-c \tau} \tau^{\mu - 1} \dd \tau = c^\mu \, \EulerE^{-(c-1)\tau^*} \gammafcnregularizedQ( \mu, y ) < c^\mu \, \EulerE^{-(c-1) y} \gammafcnregularizedQ( \mu, y ),
\end{equation}
where the second equality follows from the mean value theorem for some $\tau^* > y$. Then
\begin{equation} \label{eq:auxiliary.u.estimate}
u( 1/K ) < \frac{c^\mu}{K} \, \EulerE^{-(c-1) \inversegammafcnregularizedQ( \mu, 1/K )} = o( K^{-1} ) \qquad \text{as $K \to \infty$.}
\end{equation}

\begin{lem} \label{lem:auxiliary.main.integral}
Let $\mu > 0$ and $c > 1$. Then
\begin{equation*}
\int_{1/K}^1 u( x ) \dd x = \incompletebetafcnregularized_{1/(1+c)}( \mu, \mu ) + o( K^{-2} ) \qquad \text{as $K \to \infty$.} 
\end{equation*}
\end{lem}

\begin{proof}
The substitution $x = \gammafcnregularizedQ( \mu, y )$, $\dd x = - ( 1 / \gammafcn( \mu ) ) \EulerE^{-y} y^{\mu-1} \dd y$, and changing to an integral representation  gives
\begin{equation*}
\int_{0}^1 u(x) \dd x = \frac{1}{\gammafcn( \mu )} \int_{0}^\infty \gammafcnregularizedQ( \mu, c \, y ) \, \EulerE^{-y} y^{\mu-1} \dd y = \frac{1}{[ \gammafcn( \mu ) ]^2} \int_0^\infty \int_{c y}^\infty \EulerE^{-(t + y)} \left( t y \right)^{\mu - 1} \dd t \dd y.
\end{equation*}
The change of variables $v = t + y$ and $w = t y $ with the Jacobian $1 / \sqrt{v^2-4w}$ and computation with the help of Mathematica gives
\begin{align*}
\int_{0}^1 u(x) \dd x 
&= \frac{1}{[ \gammafcn( \mu ) ]^2} \int_0^\infty \int_0^{\frac{c}{(1+c)^2} u^2} \frac{\EulerE^{-v} w^{\mu-1}}{\sqrt{v^2-4w}} \dd v \dd w \\
&= \frac{1}{2} - \frac{\gammafcn( \mu + 1/2 )}{\sqrt{\pi} \, \gammafcn( \mu )} \, \frac{c-1}{c+1} \, \Hypergeom{2}{1}{1-\mu,1/2}{3/2}{\left( \frac{c-1}{c+1} \right)^2}.
\end{align*}
The Gauss hypergeometric function can be expressed in terms of an incomplete beta function (\cite[Eq.~8.17.7]{NIST:DLMF}) which can be turned into its regularized form; i.e., using relations for the regularized incomplete beta function (see \cite[Sec.~8.17]{NIST:DLMF})
\begin{equation*}
\int_{0}^1 u(x) \dd x = \frac{1}{2} - \frac{1}{2} \incompletebetafcnregularized_{(c-1)^2/(c+1)^2}\big( \frac{1}{2}, \mu \big) = \frac{1}{2} \incompletebetafcnregularized_{4 c/(c+1)^2}\big( \mu, \frac{1}{2} \big) = \incompletebetafcnregularized_{1/(1+c)}( \mu, \mu ).
\end{equation*}
Hence
\begin{equation*}
\int_{1/K}^1 u(x) \dd x = \incompletebetafcnregularized_{1/(1+c)}( \mu, \mu ) - \frac{1}{\gammafcn( \mu )} \int_{\inversegammafcnregularizedQ(\mu, 1/K)}^\infty \gammafcnregularizedQ( \mu, c \, y ) \, \EulerE^{-y} y^{\mu-1} \dd y.
\end{equation*}
Using the estimate \eqref{eq:auxiliary.estimate} with $y = \inversegammafcnregularizedQ(\mu, 1/K)$, the subtracted integral can be bounded as 
\begin{align*}
0 
&< \frac{1}{\gammafcn( \mu )} \int_{\inversegammafcnregularizedQ(\mu, 1/K)}^\infty \gammafcnregularizedQ( \mu, c \, y ) \, \EulerE^{-y} y^{\mu-1} \dd y < \frac{c^\mu}{K} \int_{\inversegammafcnregularizedQ(\mu, 1/K)}^\infty e^{-c \, y} y^{\mu-1} \dd y = \frac{u(1/K)}{K} = o( K^{-2} )
\end{align*}
as $K \to \infty$, where the last step follows from \eqref{eq:auxiliary.u.estimate}.
\end{proof}

Let $\widetilde{\bernoulliB}_n(x)$ denote the periodic Bernoulli function of degree $n$ given by $\widetilde{\bernoulliB}_n( x ) = \bernoulliB_n( x )$ for $0 \leq x < 1$ and $\widetilde{B}_n( x + 1 ) = \widetilde{B}_n( x )$ for $x \in \mathbb{R}$, where $\bernoulliB_n(x)$ are Bernoulli polynomials. In particular, $\bernoulliB_0( x ) \equiv 1$ and $\bernoulliB_1( x ) = x - 1/2$. 

\begin{lem} \label{lem:auxiliary.remainder.integral}
Let $\mu > 0$ and $c > 1$. Then
\begin{equation*}
\int_{1/K}^1 u^\prime( x ) \, \widetilde{\bernoulliB}_1( K x ) \, \dd x = \frac{c^\mu}{12 K} + o( K^{-1} ) \qquad \text{as $K \to \infty$.} 
\end{equation*}
\end{lem}

\begin{proof}
By dividing the integration domain and using the periodicity of $\widetilde{\bernoulliB}_1( K x )$, we get
\begin{equation*}
\int_{1/K}^1 u^\prime( x ) \, \widetilde{\bernoulliB}_1( K x ) \, \dd x = \sum_{k=1}^{K-1} \int_{k/K}^{(k+1)/K} u^\prime( x ) \, \widetilde{\bernoulliB}_1( K x ) \, \dd x = \frac{1}{K} \sum_{k=1}^{K-1} \int_0^1 u^\prime\big( \frac{k+y}{K} ) \, \bernoulliB_1( y ) \, \dd y.
\end{equation*}
By symmetry of $\bernoulliB_1( y ) = y - 1/2$ about $y = 1/2$, we can write
\begin{equation*}
\int_{1/K}^1 u^\prime( x ) \, \widetilde{\bernoulliB}_1( K x ) \, \dd x = \frac{1}{K} \int_0^{1/2} H_K( y ) \left( \frac{1}{2} - y \right) \dd y, 
\end{equation*}
where
\begin{equation*}
H_K( y ) \DEF \sum_{k=1}^{K-1} \left[ u^\prime\big( \frac{k+1-y}{K} ) - u^\prime\big( \frac{k+y}{K} ) \right] = u^\prime\big( 1 - \frac{y}{K} ) - u^\prime\big( \frac{1-y}{K} ) - G_K( y ), 
\end{equation*}
and
\begin{equation*}
G_K( y ) \DEF \sum_{k=1}^{K-1} \left[ u^\prime\big( \frac{k+y}{K} ) - u^\prime\big( \frac{k-y}{K} ) \right].
\end{equation*}
For $y \in [0,1/2]$, the square-bracketed expressions in the last two displayed formulas are non-negative as can be seen from the positivity of 
\begin{equation*}
u^{\prime\prime}( x ) = \gammafcn( \mu ) \left( c - 1 \right) c^\mu \, \EulerE^{-(c-2) \inversegammafcnregularizedQ(\mu, x)} \left[ \inversegammafcnregularizedQ(\mu, x) \right]^{1-\mu}. 
\end{equation*}
A mean value theorem application yields
\begin{align*}
G(y) 
&\leq \int_{1}^K \left[ u^\prime\big( \frac{k+y}{K} ) - u^\prime\big( \frac{k-y}{K} ) \right] \dd y = K \left( \int_{(1+y)/K}^{1 - (1-y)/K} - \int_{(1-y)/K}^{1-(1+y)/K} \right) u^\prime( x ) \, \dd x \\
&= K \int_{1-(1+y)/K}^{1 - (1-y)/K} u^\prime( x ) \, \dd x - K \int_{(1-y)/K}^{(1+y)/K} u^\prime( x ) \, \dd x \\
&= 2 y u^\prime( z^* ) - 2 y u^\prime( z^{**} ) \leq 2y \, c^\mu
\end{align*}
for some $z^* \in ( 1 - (1+y)/K, 1 - (1-y)/K )$ and some $z^{**} \in ( (1-y)/K, (1+y)/K )$. So
\begin{equation*}
H_K( y ) \geq u^\prime\big( 1 - \frac{y}{K} ) - u^\prime\big( \frac{1-y}{K} ) - 2y \, c^\mu \geq  c^\mu \left( 1 - 2y \right) - \left( c^\mu - u^\prime\big( 1 - \frac{1}{K} \big) + u^\prime\big( \frac{1}{K} \big) \right) 
\end{equation*}
for $0 \leq y \leq 1/2$. In a similar way (details are left to the reader), one can show that
\begin{equation*}
H_K( y ) \leq c^\mu \left( 1 - 2y \right) + \left( c^\mu - u^\prime\big( 1 - \frac{2}{K} \big) + u^\prime\big( \frac{3}{K} \big) \right), \quad 0 \leq y \leq 1/2.
\end{equation*}
Hence
\begin{equation*}
\int_{1/K}^1 u^\prime( x ) \, \widetilde{\bernoulliB}_1( K x ) \, \dd x = \frac{c^\mu}{K} \int_0^{1/2} \left( 1 - 2 y \right) \left( \frac{1}{2} - y \right) \dd y + o( K^{-1} ) = \frac{c^\mu}{12 K} + o( K^{-1} ) 
\end{equation*}
as $K \to \infty$.
\end{proof}

We are ready now to prove Lemma~\ref{lem:Lambda.K}.

\begin{proof}[Proof of Lemma~\ref{lem:Lambda.K}]
By the Euler-MacLaurin summation formula we have
\begin{equation*}
\Lambda_K = \frac{1}{K} \sum_{k=1}^K u\big( \frac{k}{K} \big) = \int_{1/K}^1 u( x ) \dd x + \frac{1}{2K} \left( u\big( \frac{1}{K} \big) + u\big( 1 \big) \right) + \frac{1}{K} \mathcal{R}_K,
\end{equation*}
where the remainder is given by (see Lemma~\ref{lem:auxiliary.remainder.integral})
\begin{equation*}
\frac{1}{K} \mathcal{R}_K = \frac{1}{K} \int_{1/K}^1 u^\prime( x ) \, \widetilde{\bernoulliB}_1( K x ) \, \dd x = \frac{c^\mu}{12 K^2} + o( K^{-2} ) \qquad \text{as $K \to \infty$.}
\end{equation*}
Furthermore, since $u(1) = 1$, the estimate \eqref{eq:auxiliary.u.estimate} gives
\begin{equation*}
\frac{1}{2K} \left( u\big( \frac{1}{K} \big) + u\big( 1 \big) \right) = \frac{1}{2K} + o( K^{-2} ) \qquad \text{as $K \to \infty$.}
\end{equation*}
Finally, Lemma~\ref{lem:auxiliary.main.integral} gives 
\begin{equation*}
\int_{1/K}^1 u( x ) \dd x = \incompletebetafcnregularized_{1/(1+c)}( \mu, \mu ) + o( K^{-2} ) \qquad \text{as $K \to \infty$.}
\end{equation*}
This completes the proof.
\end{proof}

\bibliographystyle{abbrv}
\bibliography{bibliographya}
\end{document}